\documentclass[11pt,reqno,UTF8,twoside]{amsart}
\usepackage{mathrsfs}
\usepackage{amsfonts,amssymb,amsmath,amsthm}
\usepackage{cite}
\usepackage[colorlinks=true,citecolor=red,linkcolor=blue]{hyperref}
\usepackage{titletoc}
\usepackage{geometry}
\usepackage{bm}
\usepackage{indentfirst}
\usepackage{graphicx}
\usepackage{stmaryrd}
\usepackage{tikz}
\usepackage{ulem}
\usepackage{color,soul}
\usepackage{setspace}
\usepackage[pagewise]{lineno} 
\usepackage{float}

\usepackage{todonotes}

\numberwithin{equation}{section}
\newtheorem{theorem}{Theorem}[section]
\newtheorem{lemma}[theorem]{Lemma}
\newtheorem{corollary}[theorem]{Corollary}
\newtheorem{proposition}[theorem]{Proposition}
\newtheorem*{maintheorem}{Main Theorem}
\newtheorem{remark}[theorem]{Remark}
\theoremstyle{definition}
\newtheorem{example}[theorem]{Example}
\newtheorem{definition}[theorem]{Definition}

\newcommand{\Lip}{\operatorname{Lip}}
\newcommand{\dist}{\operatorname{dist}}
\newcommand{\dvg}{\operatorname{div}}
\newcommand{\Osc}{\operatorname{Osc}}
\newcommand{\supp}{\operatorname{supp}}

\tikzstyle{idea} = [rectangle, rounded corners, minimum width=2cm, minimum height=1cm, text centered, draw=black, align=center]
\tikzstyle{process} = [rectangle, minimum width=3cm, minimum height=1cm, text centered, draw=black, align=center]
\tikzstyle{point} = [coordinate, on grid]
\tikzstyle{arrow} = [thick,->,>=stealth]
\tikzstyle{dasharrow} = [dashed,->,>=stealth]

\makeatletter
\@namedef{subjclassname@2020}{\textup{2020} Mathematics Subject Classification}
\makeatother

\title[Local LDP for Wave Equations]{\Large L{\MakeLowercase{ocal Large Deviations for Randomly Forced\\
Nonlinear Wave Equations with Localized Damping}}}

\author[Y. Chen, Z. Liu, S. Xiang, Z. Zhang]{Y\MakeLowercase{uxuan} C\MakeLowercase{hen}, Z\MakeLowercase{iyu} L\MakeLowercase{iu}, S\MakeLowercase{hengquan} X\MakeLowercase{iang}, Z\MakeLowercase{hifei} Z\MakeLowercase{hang}}

\address[Yuxuan Chen]{School of Mathematical Sciences, Peking University, 100871, Beijing, China.}
\email{chen\_yuxuan@pku.edu.cn}

\address[Ziyu Liu]{School of Mathematical Sciences, Peking University, 100871, Beijing, China.}
\email{liuziyu@math.pku.edu.cn}

\address[Shengquan  Xiang]{School of Mathematical Sciences, Peking University, 100871, Beijing, China.}
\email{shengquan.xiang@math.pku.edu.cn}

\address[Zhifei  Zhang]{School of Mathematical Sciences, Peking University, 100871, Beijing, China.}
\email{zfzhang@math.pku.edu.cn}

\begin{document}

	

    \vspace{5mm}
	\begin{abstract}
        We study the large deviation principle (LDP) for locally damped nonlinear wave equations perturbed by a bounded noise. When the noise is sufficiently non-degenerate, we establish the LDP for empirical distributions with lower bound of a local type. The primary challenge is the lack of compactness due to the absence of smoothing effect. This is overcome by exploiting the  asymptotic compactness for the dynamics of waves, introducing the concept of asymptotic exponential tightness for random measures, and establishing a new LDP approach for random dynamical systems.
    \end{abstract}
 
    \maketitle

    \setcounter{tocdepth}{2}
	\tableofcontents
    \section{Introduction}
    
    The {\it large deviation principle} (LDP) is an important topic in mathematics and physics. Many scientific inquiries, such as those in statistical mechanics, dynamical systems, and applied probability, are intricately connected with this theory; see monographs \cite{Ell-85, DS89, DZ10}. The LDP evaluates small probabilities of rare events on the exponential scale. Roughly speaking, a sequence of probability measures $(\mu_n)_{n\in \mathbb{N}}$ satisfies the LDP with rate function $I$, if    
    \[\mu_n(dx)\approx e^{-nI(x)}\, dx\quad \text{as }n\to \infty.\]
    
    \vspace{2cm}
    The motivation of this paper is twofold:
    
    \begin{enumerate}
        \item To extend the study on statistical behaviours of randomly forced dispersive PDEs, building upon the prior paper \cite{LWXZZ24}. The highlight is the local LDP for empirical distributions of weakly dissipative nonlinear wave equations. 
        
        \item To explore sharp sufficient conditions for LDP, specifically by advancing from exponential tightness to an asymptotic formulation. This new concept is a consequence of asymptotic compactness, from the dynamics of wave equations.
    \end{enumerate}

    \subsection{Main result}
    The nonlinear wave equation under consideration is formulated as follows:
     \begin{equation}
        \begin{cases}
            \boxempty u+a(x)\partial_t u+u^3=\eta(t,x),\quad x\in D,\\
            u|_{\partial D}=0,\\
            u[0]=(u_0,u_1)=:\boldsymbol{u},
        \end{cases}\label{waveequation}
    \end{equation} 
    where $D$ is a smooth bounded domain in $\mathbb{R}^3$, with boundary $\partial D$ and outer normal vector $n(x)$, the notation $\boxempty:=\partial_{tt}^2-\Delta$ denotes the d'Alembert operator, and  $u[t]:=(u,\partial_t u)(t)$. Our settings for the damping coefficient $a(x)$ and random noise $\eta(t,x)$ are stated in $\mathbf{(S1)}$ and $\mathbf{(S2)}$ below.

    The phase space of equation \eqref{waveequation} is the energy space $\mathcal H:=H_0^1(D)\times L^2(D)$. 
    Let $(\lambda_j)_{j\in\mathbb{N}}$ be the eigenvalues of $-\Delta$, arranged in the increasing order with multiplicity counted. The eigenvectors corresponding to $\lambda_j$ are denoted by $e_j$, which form an orthonormal basis of $L^2(D)$. The sequence $(\alpha_k)_{k\in\mathbb{N}}$ denotes a smooth orthonormal basis of $L^2(0,1)$, which induces an orthonormal basis of $L^2(0,T)$ by taking  $\alpha_k^{\scriptscriptstyle T}(t)=\frac{1}{\sqrt{T}}\alpha_k(\frac{t}{T})$.

    \vspace{3mm}

    Below is to introduce the notion of $\Gamma$-type domain, which is initially used by Lions \cite{Lions-88}.  
    \begin{definition}\label{Def-Gamma}
    A $\Gamma$-type domain is a subdomain of $D$ in the form
    $$
    N_\delta(x_0):=\{
    x\in D:|x-y|<\delta\  \text{ for some }y\in\Gamma(x_0)\},
    $$
    where $x_0\in\mathbb R^3\setminus \overline D,\,\delta>0$ and $\Gamma(x_0)=\{x\in\partial D:(x-x_0)\cdot n(x)>0\}$. 
    \end{definition}

    This geometric setting is involved both in the localization of $a(x)$ and the structure of $\eta(t,x)$:

    \begin{enumerate}
    \item[$\mathbf{(S1)}$] ({\bf Localized structure}) {\it The function $a(\cdot)\in C^\infty(\overline{D})$ is non-negative, and there exists a $\Gamma$-type domain $N_\delta(x_0)$ and a constant $a_0>0$ such that
    \begin{equation*}
    a(x)\geq a_0,\quad \forall\,x\in N_\delta(x_0).
    \end{equation*}
    Meanwhile, let $\chi(\cdot)\in C^\infty(\overline{D})$ satisfy that there exists a $\Gamma$-type domain $N_{\delta'}(x_1)$ and a constant $\chi_0>0$ such that
    \begin{equation*}
    \chi(x)\geq \chi_0,\quad\forall\,x\in  N_{\delta'}(x_1).
    \end{equation*}
    }
    \end{enumerate}

   \noindent In particular, $\color{black}\mathbf{(S1)}$ would determine an intrinsic quantity $\mathbf{T}= \mathbf{T}(D, a, \chi)>0$, serving as a lower bound for time spread of the noise $\eta(t,x)$; see  \cite[Section 6]{LWXZZ24} for a rigorous formulation. 
   
    \vspace{1mm}
    
    \begin{enumerate}
    \item[$\mathbf{(S2)}$]  
    {\it 
    Let $\boldsymbol{\rho}=\{\rho_{jk}:j,k\in\mathbb{N}\}$ be a sequence of probability density functions supported by $[-1,1]$, which is $C^1$ and satisfies $\rho_{jk}(0)>0$. }
    \end{enumerate}

      \vspace{3mm}
    Given any $T>0$ and $\{b_{jk}:j,k\in\mathbb{N}\}$, a sequence of non-negative numbers,  the random noise $\eta(t,x)$ is specified as
    \begin{equation}\label{noise-structure1}
    \begin{aligned}
    &\eta(t,x)= \eta_n(t-nT,x)\quad \text{for }t\in (nT,(n+1)T)\text{ and }n\in\mathbb{N}_0,\\ 
    &\displaystyle\eta_n(t,x)=\chi(x)\sum_{j,k\in\mathbb{N}}b_{jk}\theta^{n}_{jk}\alpha^{\scriptscriptstyle T}_k(t)e_j(x)\quad \text{for }t\in (0,T),
    \end{aligned}
    \end{equation}
    where $\theta_{jk}^{n}$ are independent random variables such that $(\theta_{jk}^{n})_{n\in\mathbb{N}_0}$ has a common density $\rho_{jk}$.

    \vspace{3mm}

    Under the above settings, the solution $u[\cdot]$ at times $nT$ for $n\in\mathbb{N}_0$, namely $\boldsymbol{u}_n:=u[nT]$, naturally form a Markov process. Accordingly, our research regarding the evolution of random process $\boldsymbol{u}_n$ starts with the following two questions: 
    \vspace{1mm}
    \begin{enumerate}
        \item {\it Does the process $\boldsymbol{u}_n$ converge to a unique equilibrium distribution? }
        \vspace{1mm}
        \item {\it If so, how to describe the fluctuation from equilibrium in terms of LDP? }
    \end{enumerate}
    \vspace{1mm}
    It should be mentioned that both problems have been extensively studied for parabolic systems, while few results are available for hyperbolic equations. 

    \vspace{3mm}
    Recently, the first question has been answered affirmatively in \cite[Theorem B]{LWXZZ24}  for the random nonlinear wave equation \eqref{waveequation}. More precisely, under the assumptions as in the \hyperlink{MT}{\color{black}Main Theorem} below, the random process $(\boldsymbol{u}_n)_{n\in\mathbb{N}_0}$ is {\it exponential mixing} in the following sense:  There exists a unique invariant probability measure $\mu_*$ on $\mathcal{H}$, such that  
    \begin{equation*}
        \|\mathscr{D}(\boldsymbol{u}_n)-\mu_*\|_L^* \le C(1+E(\boldsymbol{u})) e^{-\gamma n}\quad\text{for any } \boldsymbol{u}\in\mathcal{H}\text{ and }n\in\mathbb{N}_0,
    \end{equation*}
    where $C,\gamma>0$ are constants, $\|\cdot\|_L^*$ is the dual-Lipschitz distance, $\mathscr{D}(\boldsymbol{u}_n)$ is the law of $\boldsymbol{u}_n$, and $E:\mathcal H\rightarrow [0,\infty)$ denotes the energy functional defined by
     \begin{equation}
        E(\boldsymbol{u}):=\int_D \frac{1}{2}|\nabla u_0|^2+ \frac{1}{2}|u_1|^2+\frac{1}{4} |u_0|^4\quad \text{for }\boldsymbol{u}=(u_0,u_1).\label{energy}
    \end{equation}

    \vspace{3mm}

   The main result of this paper concerning the second question is contained as follows, establishing a local form of large deviations for  $\boldsymbol{u}_n$. Recall that a rate function refers to a lower semicontinuous function $I$ taking values in $[0,\infty]$, and $I$ is good if its level set $\{I\le a\}$ is compact for any $a\in [0,\infty)$; see, e.g.~\cite{DZ10}.

    \begin{maintheorem}
   \hypertarget{MT}{Suppose} that $a(x),\,\chi(x),\,\boldsymbol{\rho}$ satisfy settings $\mathbf{(S1)}$ and $\mathbf{(S2)}$. Given any $T>\mathbf{T}(D,a,\chi)$ and ${B}_0>0$, there exists a constant $N=N(T,B_0)\in\mathbb{N}$ such that if the sequence $\{b_{jk}:j,k\in\mathbb{N}\}$ in {\rm(\ref{noise-structure1})} satisfies
    \begin{equation}\label{noise-structure2}
    \sum_{j,k\in\mathbb{N}}b_{jk}\lambda_j^{2/7}\|\alpha_k\|_{_{L^\infty(0,1)}} \leq {B}_0T^{1/2}\quad\text{and}\quad  b_{jk}\neq0\ \text{ for } 1\leq j,k\leq N.
    \end{equation}
    Then for any bounded Lipschitz function $f\colon \mathcal{H}\to \mathbb{R}$ with $f(0)\not =\langle f,\mu_*\rangle$, there exists a good rate function $\mathcal{I}^f\colon\mathbb{R}\to [0,\infty]$ and $\varepsilon>0$, such that
    \begin{equation}
            \lim_{n\to \infty} \frac{1}{n}\log \mathbb{P}_{\boldsymbol{u}} \left(\frac{1}{n}\sum_{k=1}^{n}f(\boldsymbol{u}_k)\in O\right)=-\inf \{\mathcal{I}^f(p):p\in O\}\label{def_level1ldp}
        \end{equation}
        for any $\boldsymbol{u}\in \mathcal{H}$ and open set $O\subset [\langle f,\mu_*\rangle-\varepsilon,\langle f,\mu_*\rangle+\varepsilon]$, where $\langle f,\mu_*\rangle:=\int_{\mathcal{H}}f\, d\mu_*$. In addition, the convergence in \eqref{def_level1ldp} holds locally uniformly with respect to $\boldsymbol{u}\in \mathcal{H}$.
    \end{maintheorem}

    \subsection{Strategy and ingredients}
    In order to prove the \hyperlink{MT}{\color{black}Main Theorem}, we establish a new random dynamical system (RDS) approach. We believe this RDS approach is applicable to the LDP theory for a wide range of dispersive equations, encompassing Schr\"odinger and KdV equations. The lack of compactness represents the major challenge in various aspects. In the context of parabolic systems, the solutions gain extra regularity along trajectories, and enter a compact subset of the original phase space. Nevertheless, for hyperbolic equations the regularity of solutions is preserved. To overcome this difficulty, we put forward a three-step strategy, as illustrated in Figure~\ref{fig1}.

\setlength{\abovecaptionskip}{0cm}
    
    \begin{figure}[H]
\centering
\begin{tikzpicture}[node distance=2cm]
\footnotesize {
\node (Kifer) [process] {\hyperlink{AET}{$\color{black}\mathbf{(AET)}$}-based\\
Kifer's Criterion} 
node [below of = Kifer, yshift=32] {Theorem~\ref{Thm Kifer}};
\node (Model) [process, right of=Kifer, xshift=3cm] {
\hyperlink{AC}{$\color{black}\mathbf{(AC)}$}-based\\
RDS Approach}
node [below of = Model, yshift=32] {Theorem~\ref{Thm LDP}};
\node (Main Theorem) [process, right of=Model, xshift=3cm] {Nonlinear Wave\\
Local LDP}
node [below of = Main Theorem, yshift=32] {\hyperlink{MT}{\color{black}Main Theorem}};

\node [below of = Model, yshift=12] {\scriptsize{\hyperlink{AET}{$\color{black}\mathbf{(AET)}$}=asymptotic exponential tightness\qquad\qquad \hyperlink{AC}{$\color{black}\mathbf{(AC)}$}=asymptotic compactness}};

\draw [arrow] (Kifer) -- (Model) node[above,midway] {Sec.~\ref{subsec_prooflevel2LDP}};
\draw [arrow] (Model) -- ({Main Theorem}) node[above,midway] {Sec.~\ref{subsec_maintheorem}};}
\end{tikzpicture}
\caption{}
\label{fig1}
\end{figure}

    \vspace{-0.3cm}

    We briefly summarize the new ingredients in each stage as follows:

    \begin{enumerate}
        \item[(1)] Kifer-type criteria are sufficient conditions for the LDP of random probability measures. In \cite{Kifer-90,JNPS-18}, exponential tightness plays a crucial role for such criteria, which requires the measures to concentrate on compact sets. To get rid of the reliance on compactness, we introduce the notion of {\bf asymptotic exponential tightness} and establish a novel criterion of this type.
        
        \item[(2)] The lack of compactness also causes obstacles to the LDP for RDS model. In compensation, we exploit {\bf asymptotic compactness} from the dynamical point of view and provide a new RDS model that guarantees the local LDP.
      
        \item[(3)] Other difficulties occur from wave equations, also related to the loss of smoothing effect:
        
        \begin{enumerate}
            \item[(i)] Both the damping and noises are localized in space. We invoke multiplier methods, while standard energy methods are no longer applicable.

            \item[(ii)] Sobolev-critical cubic nonlinearity is tackled, benefiting from the control theory in place of the Foias--Prodi type estimates. 
        \end{enumerate}
    \end{enumerate}

    \vspace{3mm}
   Below is a more detailed explanation of each stage. 
  
  The RDS approach is inspired by \cite{Kifer-90,JNPS-15,LWXZZ24}. Our new criterion reveals that the following three hypotheses lead to local LDP. For the precise definitions, see Section~\ref{subsec_mixing}.

    \begin{itemize}
        \item [\tiny$\bullet$] {\bf Asymptotic compactness.} { A compact set $Y\subset \mathcal{H}$ attracts the sequence $\boldsymbol{u}_n$. }
        
        \item[\tiny$\bullet$] {\bf Irreducibility.} { A specific point $z\in Y$ is accessible from $Y$. }
         
        \item[\tiny$\bullet$] {\bf Coupling condition.} { Two processes issuing from $Y$ are possible to become closer.}
    \end{itemize}
    A new ingredient in this approach is asymptotic compactness, which has been coined recently in \cite{LWXZZ24} for mixing of wave equations. The merit of asymptotic compactness is that, the process tends to the compact set $Y$ without necessarily entering it. Meanwhile, irreducibility and coupling condition have been exploited in \cite{JNPS-15,JNPS-18} for the study of LDP for parabolic systems. Note that the latter two are typical hypotheses for mixing, e.g.~\cite{EM-01,HM-06,HM-08,Shi-08,KS-12,Shi-15}.

    \vspace{-1.5mm}
    
    \begin{remark}
    Heuristically, by virtue of mixing, the process tends to $\supp(\mu_*)$ with high probability, which suggests asymptotic compactness to be a sharp condition for the LDP.
    \end{remark}

    Kifer's large-deviation criterion for random probability measures has been exploited in \cite{JNPS-15,JNPS-18} to derive the LDP for parabolic systems. Based on the novel idea of asymptotic exponential tightness, we will set up a variant of this criterion. This together with aforementioned three hypotheses leads to the local LDP for RDS model. To briefly compare this concept with the standard notion of exponential tightness (see also Remark~\ref{Rmk_AETvsET}), recall a sequence $\mu_n\in \mathcal{P}(\mathcal{X})$ is exponentially tight, if for any $l>0$, there exists a compact set $K_l\subset \mathcal{X}$, such that    
    \[\limsup_{n\to \infty} \frac{1}{n}\log \mu_n(K_l^C)\le -l.\]
    Meanwhile, asymptotic exponential tightness asks for a weaker relation:
    \[\limsup_{n\to \infty} \frac{1}{n}\log \mu_n(B_{\mathcal{P(X)}}(K_l,r)^C)\le -l\quad \text{for any }r>0.\]
    The precise definition is provided in \eqref{AET}. In our RDS setting, asymptotic exponential tightness for empirical distributions is a consequence of the asymptotic compactness hypothesis.

    \vspace{3mm}
    A large portion of \cite{LWXZZ24} has been devoted to the verification that, under the assumptions of the \hyperlink{MT}{\color{black}Main Theorem}, the process $\boldsymbol{u}_n$ fits into the RDS setting. Indeed, the study of LDP is divided into several interconnected core topics in analysis; see also \cite[Section 1.3]{LWXZZ24}. Asymptotic compactness corresponds to attractors in dynamical systems and the dynamics of the evolution. Although $\boldsymbol{u}_n$ merely stays in $\mathcal{H}$, there exists a compact set $Y$ such that 
    \[\dist_{\mathcal{H}}(\boldsymbol{u}_n,Y)\le C(1+E(\boldsymbol{u})) e^{-\kappa n}.\]
    Irreducibility pertains to the global stability (or dissipation) of the system. The damping effect implies exponential decay once the noise vanishes. Coupling condition is associated with the stabilization from control theory.  Therefore, our approach is applicable to wave equations.

    \vspace{3mm}
    
    We end the introduction with a brief historical note on the LDP theory and randomly forced PDEs. The foundation of the modern theory on the large deviations for Markov processes is built by Donsker and Varadhan \cite{DV-75}. Then Kifer \cite{Kifer-90} establishes a general LDP criterion for random probability measures. These results mainly apply to compact metric spaces. In the context of non-compact spaces, the concept of exponential tightness is captured by Deuschel and Stroock \cite{DS89} to improve the LDP upper bound from compact sets to closed sets. Nevertheless, in certain scenarios, this condition is difficult or even impossible to verify, particularly when  investigating the LDP for hyperbolic equations, as previously explained.
    
   The Donsker--Varadhan type LDP for random PDEs has been studied in the past two decades. The first results for the  Navier--Stokes and Burgers equations with white-in-time forces are established by Gourcy \cite{Gou07-NS,Gou07-Burgers}, where the proofs are based on a general sufficient condition established by Wu \cite{Wu01}.  Jakši\'{c}, Nersesyan, Pillet and Shirikyan \cite{JNPS-15,JNPS-18} derive the LDP for a family of dissipative PDEs including the Navier--Stokes system, perturbed by non-degenerate random kick forces. The proofs are based on Kifer-type criteria and the long-time behaviour of Feymann--Kac semigroups. Nersesyan, Peng and Xu \cite{NPX-23} recently establish the LDP for the Navier--Stokes equations driven by a degenerate white-in-time noise. In \cite{JNPS-21}, a controllability method is used to study the LDP for Lagrangian trajectories of the Navier--Stokes system.

    In contrast, research on hyperbolic equations is almost a vacuum. To the best of our knowledge, the only existing literature on this topic is \cite{MN-18} by Martirosyan and Nersesyan, which treats the local LDP for damped wave equations with Sobolev-subcritical nonlinearity perturbed by a white noise, under the additional regularity assumption that $\boldsymbol{u}\in \mathcal{H}^s$ for some $s>0$.
    
    Finally, we also refer the reader to the literature, e.g.~\cite{FW-84,CR04,SS06,CM10,BCF15,Mar17} for various results of the Freidlin--Wentzell type LDP, and \cite{Bourgain-96,BT-08,GOT-22,BDNY-24,GKO-24} for other related topics in random dispersive equations.

    \vspace{3mm}
    
    This paper is organized as follows. In Section~\ref{sec_model} we first present the RDS setting and results on the local LDP, and then turn to the proof of the \hyperlink{MT}{\color{black}Main Theorem} on wave equations. Two more applications are provided: one indicates that the localness can be intrinsic, and the other, the Navier--Stokes system, suggests how the localness can be removed with further properties. Section~\ref{sec_proofofabstractldp} is devoted to the proof of the LDP in our abstract setting. A new feature in the proof is a variant of Kifer's criterion based on asymptotic exponential tightness. Finally, we exhibit some auxiliary results and proofs in the Appendix.

    \vspace{3mm}
    \noindent{\bf Notation.} 
    Let $(X,d)$ be a Polish space, i.e.~complete separable metric space, whose Borel $\sigma$-algebra is $\mathcal{B}(X)$. The interior and closure of $B\in \mathcal{B}(X)$ are written as $B^\circ$ and $\overline{B}$, respectively. We denote by $B_X(x,r)$ the open ball in $X$ centered at $x\in X$ with radius $r>0$. For $A\subset X$, the same notation represents the $r$-neighbourhood of $A$:
        \[B_X(A,r):=\{x\in X:\dist_X(x,A)<r\},\]
    where $\dist_X(x,A)$ stands for the distance from $x$ to $A$.

    By $C_b(X)$, we denote the Banach space of all bounded continuous functions on $X$, equipped with the supremum norm $\|\cdot\|_{\infty}$. For $f\in C_b(X)$,  its oscillation is denoted by $\Osc(f)$. $L_b(X)$ stands for the subspace of bounded Lipschitz functions. The Lipschitz norm of $f\in L_b(X)$ is $\|f\|_{L}:=\|f\|_{\infty}+ \Lip(f)$, where $\Lip(f):= \sup_{x\neq y}\frac{|f(x)-f(y)|}{d(x,y)}$. When $X$ is compact, we shall omit the subscripts for $C_b(X)$ and $L_b(X)$, and simply write $C(X)$ and $L(X)$, respectively.

    The space of finite signed Borel measures is denoted by $\mathcal{M}(X)$, which is endowed with the weak-* topology. The subset of non-negative Borel measures is $\mathcal{M}_+(X)$, and $\mathcal{P}(X)$ consists of probability measures. We write $\supp (\mu)$ for the support of $\mu\in \mathcal{M}(X)$.  For $f\in C_b(X)$ and $\mu\in\mathcal{P}(X)$, the integration is written as $\langle f,\mu\rangle:=\int_X f\, d\mu$.
    
     Finally, the weak-* topology on $\mathcal P(X)$ is compatible with the dual-Lipschitz distance:
        \begin{equation*}
            \|\mu-\nu\|_L^*:=\sup_{\|f\|_{L}\le 1}|\langle f,\mu\rangle-\langle f,\nu\rangle|\quad \text {for } \mu,\nu\in\mathcal{P}(X).
        \end{equation*}

    \section{Model and results}\label{sec_model}

    We prove the \hyperlink{MT}{\color{black}Main Theorem} for the random wave equation~\eqref{waveequation} in Section~\ref{subsec_maintheorem}. To this end, we first present our RDS model based on asymptotic compactness in Section~\ref{subsec_mixing} that ensures a local form of LDP, also serving as the main contributions of this paper. The precise statements are contained in Theorem~\ref{Thm LDP} and Corollary~\ref{Cor level1LDP}, whose proofs are intricate and postponed to Section~\ref{sec_proofofabstractldp}. 
    In addition, we present two more applications of our RDS setting: a toy model indicates that the localness can be intrinsic, while the 2D Navier--Stokes system suggests how to remove the localness under additional properties.
    
    \vspace{8mm}
    \subsection{RDS setting and large deviations}\label{subsec_mixing}

    \subsubsection{Review on mixing}
    
    We begin with our RDS setting and a quick review of the results on exponential mixing. Let $(\mathcal{X},d)$ be a Polish space, and $(\mathcal{E},d_{\mathcal E})$ a compact metric space. The sequence $(\xi_n)_{n\in\mathbb{N}_0}$ stands for $\mathcal E$-valued i.i.d.~random variables. Without loss of generality, assume $\mathcal E=\supp(\mathscr{D}(\xi_0))$. Given the locally Lipschitz mapping $S\colon \mathcal{X}\times \mathcal E\to \mathcal{X}$, we intend to consider the RDS defined by
    \begin{equation}
        x_n=S(x_{n-1},\xi_{n-1}),\quad x_0=x\in \mathcal{X}.\label{RDS}
    \end{equation}
    
    In order to indicate the initial condition and random inputs, we also denote
    \[x_n=S_{n}(x;\xi_0,\dots,\xi_{n-1})=S_{n}(x;\boldsymbol{\xi})\]
    with $\boldsymbol{\xi}:=(\xi_n)_{n\in\mathbb{N}_0}$. Moreover, given a sequence $\boldsymbol{\zeta}=(\zeta_n)_{n\in\mathbb{N}_0}\in \mathcal E^{\mathbb{N}_0}$, we denote by
    \[S_n(x;\zeta_0, \dots, \zeta_{n-1})=S_{n}(x;\boldsymbol{\zeta})\]
    the corresponding deterministic process defined by replacing $\xi_n$ with $\zeta_n$ in \eqref{RDS}.

    By virtue of the continuity of $S$ and assumption that $\xi_n$ are i.i.d., the RDS \eqref{RDS} defines a Feller family of discrete-time Markov processes in $\mathcal{X}$; see, e.g.~\cite[Section 1.3]{KS-12}. In the context, we denote the corresponding Markov family by $\mathbb{P}_x$, the expected values by $\mathbb{E}_x$, and the Markov transition functions by $P_n(x,\cdot )$, i.e.,
    \[P_n(x,B)=\mathbb{P}_x(x_n\in B)\quad \text{for }B\in\mathcal{B}(\mathcal{X}).\]
    The standard notation for the Markov semigroup $P_n\colon C_b(\mathcal{X})\rightarrow C_b(\mathcal{X})$ is employed:
    \begin{equation*}
        P_nf(x):= \int_{\mathcal{X}} f(y)P_n(x,dy).
    \end{equation*}
    And $P^*_n\colon \mathcal{P}(\mathcal{X})\rightarrow \mathcal{P}(\mathcal{X})$ refers to the dual semigroup. For the sake of simplicity, we write $P=P_1$ and $P^*=P_1^*$. A probability measure $\mu$ is called \textit{invariant} for this RDS if $P^*\mu=\mu$. We say $Y\subset \mathcal{X}$ is invariant, if 
    \[S(y,\zeta)\in Y\quad \text{for any }y\in Y\text{ and }\zeta\in \mathcal E.\]
        
    Below is a list of hypotheses regarding the abstract criteria for exponential mixing and LDP. 

    \begin{itemize}
        \item [\hypertarget{AC}{$\mathbf{(AC)}$}] {\bf Asymptotic compactness.} There exists a compact invariant set $Y\subset \mathcal{X}$, a measurable function $V\colon \mathcal{X}\to [0,\infty)$ which is bounded on bounded sets, and a constant $\kappa>0$, such that for any $x\in \mathcal{X}$, $\boldsymbol{\zeta}\in \mathcal E^{\mathbb{N}_0}$ and $n\in \mathbb{N}_0$,
        \begin{equation}
            \dist_{\mathcal{X}} (S_n(x;\boldsymbol{\zeta}),Y)\le V(x)e^{-\kappa n}.\label{asymptotic compactness}
        \end{equation}
        
        \item [\hypertarget{I}{$\mathbf{(I)}$}] {\bf Irreducibility to a point.} There exists a point $z\in Y$ such that for any $\varepsilon>0$, one can find an integer $N=N(\varepsilon)\in \mathbb{N}$ satisfying
        \begin{equation}
            \inf_{y\in Y}P_N(y,B_Y(z,\varepsilon))>0.\label{irreducibility}
        \end{equation}
         
        \item [\hypertarget{C}{$\mathbf{(C)}$}] {\bf Coupling condition on $Y$.} There exists $q\in [0,1)$ such that for any $x,x'\in Y$, the pair $(P(x,\cdot),P(x',\cdot))$ admits a coupling\footnote{For $\mu,\nu\in\mathcal{P}(\mathcal{X})$, a coupling between $\mu$ and $\nu$ is a pair of $\mathcal{X}$-valued random variables with marginal distributions equal to $\mu$ and $\nu$, respectively.} $(R(x,x'),R'(x,x'))$  on a common probability space $(\Omega,\mathcal{F},\mathbb{P})$ with
        \begin{equation}
            \mathbb{P}(d(R(x,x'),R'(x,x'))>qd(x,x'))\le g(d(x,x')),\label{squeezing}
        \end{equation}
        where $g\colon[0,\infty)\rightarrow[0,\infty)$ is a continuous increasing function satisfying
        \begin{equation}
            g(0)=0\quad\text{and}\quad \limsup\limits_{k\rightarrow\infty}\frac{1}{k}\log g(q^k)\in [-\infty,0),\label{g-def}
        \end{equation}
        and the mappings $R,R'\colon Y\times Y\times \Omega\rightarrow Y$ are measurable.
    \end{itemize}

    To formulate the mixing property, let us introduce the following concept.
    \begin{definition}\label{def_attainableset}
        For any $Z\subset \mathcal{X}$, we define the \textit{attainable set} at time $n\in \mathbb{N}$ by
        \begin{equation*}
            \mathcal{A}_n(Z):=\{S_n(x;\zeta_0,\dots,\zeta_{n-1}):x\in Z,\ \zeta_0,\dots,\zeta_{n-1}\in \mathcal E\},
        \end{equation*}
        and the attainable set $\mathcal{A}(Z):=\overline{\bigcup_{n=0}^\infty\mathcal{A}_n(Z)}$, where $\mathcal{A}_0(Z):=Z$. 
    \end{definition}

    For simplicity, let us set $\mathcal{A}:=\mathcal{A}(\{z\})$, then one can check that $\mathcal{A}\subset Y$.  A stronger version of the following theorem on mixing is established in \cite[Theorem 2.1 and Remark 2.1]{LWXZZ24}.
    
    \begin{theorem}\label{Thm Exponential mixing}
        Suppose that hypotheses \hyperlink{AC}{$\color{black}\mathbf{(AC)}$}, \hyperlink{I}{$\color{black}\mathbf{(I)}$} and \hyperlink{C}{$\color{black}\mathbf{(C)}$} are satisfied. Then the Markov process $x_n$ defined by \eqref{RDS} admits a unique invariant measure $\mu_*\in\mathcal{P}(\mathcal{X})$, and there exist constants  $C,\gamma>0$ such that for any $x\in \mathcal{X}$ and $n\in \mathbb{N}_0$,
        \begin{equation}
        \|P_n(x,\cdot)-\mu_*\|_L^*\le C(1+V(x))e^{-\gamma n}.\label{Exponential mixing}
        \end{equation}
        Moreover, the support of $\mu_*$ coincides with $\mathcal{A}$.
    \end{theorem}

    \subsubsection{Large deviation principle}\label{subsec_ldpstatement}

    We proceed to formulate the LDP associated with the RDS \eqref{RDS}. The result for test functions, as stated in the \hyperlink{MT}{\color{black}Main Theorem}, is referred to as the level-1 LDP. It is a consequence of the level-2 LDP, which means at the level of empirical distributions.
    
    To illustrate the level-2 LDP, let $x\in \mathcal{X}$ be an arbitrary initial point, denote by $\delta_{x}$ the Dirac measure concentrated at $x$, and introduce the empirical distributions
    \begin{equation*}
        L_{n,x}:=\frac{1}{n}\sum_{k=1}^n\delta_{x_k}.
    \end{equation*}
    In order to investigate the locally uniform (with respect to $x$) convergence, for any $R>0$, set
    \begin{equation}
        X_R:=\mathcal{A}(B_{\mathcal{X}}(Y,R)),\label{XR}
    \end{equation}
    Then by hypothesis \hyperlink{AC}{$\color{black}\mathbf{(AC)}$}, $X_R$ is a bounded invariant set, and
     \begin{equation*}
         \mathcal{A}\subset Y\subset X_R\subset\mathcal{X}\quad\text{for any }R>0.
     \end{equation*}
    In particular, we can restrict the phase space  of the RDS to each $X_R$. Moreover, hypotheses \hyperlink{AC}{$\color{black}\mathbf{(AC)}$}, \hyperlink{I}{$\color{black}\mathbf{(I)}$} and \hyperlink{C}{$\color{black}\mathbf{(C)}$} are still valid, with $\mathcal{X}$ replaced by $X_R$, the compactum $Y$ unchanged, and $V(x)$ in \eqref{asymptotic compactness} replaced by a constant depending on $R$.
    \vspace{3mm}
    
    Following the ideas for the study of uniform LDP from \cite{Kifer-90}, let us consider a \textit{net}\footnote{A directed set means a partially ordered set with the additional property that any finite subset admits an upper bound. And a net is a family labelled by a directed set.} of random probability measures. Define a partial order $\prec$ on $\Theta_R:=\mathbb{N}\times X_R$ by
    \begin{equation*}
        (n_1,x_1)\prec (n_2,x_2)\quad\text{if and only if}\quad n_1\le n_2.
    \end{equation*} 
    Then $(\Theta_R,\prec)$ is a directed set, and $L_{n,x}=L_\theta$ is a net indexed by $\Theta_R$.
    
    Let $r_\theta:=n$ for $\theta=(n,x)$. For any $V\in C_b(X_R)$, the pressure function $\Lambda_R(V)$ is defined by
    \begin{equation}
        \Lambda_R(V):=\limsup_{\theta\in\Theta_R}\frac{1}{r_\theta}\log \mathbb{E}e^{r_\theta\langle V,L_{\theta}\rangle}.\label{pressuredefinition}
    \end{equation}
    Equivalently,
    \[\Lambda_R (V)=\limsup_{n\to \infty} \sup_{x\in X_R}\frac{1}{n}\log \mathbb{E}_x e^{V(x_1)+\cdots +V(x_n)}.\]
    One can readily see that $\Lambda_R\colon C_b(X_R)\to \mathbb{R}$ is convex and $1$-Lipschitz, and that for any $C\in \mathbb{R}$,
    \begin{equation}
        \inf_{X_R} V\le \Lambda_R (V)\le \sup_{X_R} V,\quad \Lambda_R (C)=C\quad \text{and}\quad \Lambda_R (V+C)=\Lambda_R (V)+C.\label{Lambdaconstant}
    \end{equation}
    The Legendre transform of $\Lambda_R$ is denoted by $I_R\colon \mathcal{P}(X_R)\to [0,\infty]$, whose definition is
    \begin{equation}
        I_R(\sigma):=\sup\limits_{V\in C_b(X_R)}(\langle V,\sigma\rangle-\Lambda_R(V)).\label{ratedefinition}
    \end{equation}
    Here the inequality $I_R(\sigma)\ge 0$ follows from $\Lambda_R (0)=0$. It is well-known from convex analysis that $I_R$ is convex and lower semicontinuous. In particular, $I_R$ is a rate function on $\mathcal{P}(X_R)$. Moreover, the pressure function can be reconstructed by the duality formula
    \begin{equation}
        \Lambda_R(V)=\sup_{\sigma\in \mathcal{P}(X_R)}(\langle V,\sigma\rangle-I_R(\sigma)).\label{inverseLegendre}
    \end{equation}

    Later we will demonstrated that $I_R(\sigma)=\infty$ for any $\sigma\not \in \mathcal{P}(Y)$ (see the derivation of \eqref{D(I)subsetP(Y)}), and thus the duality relation \eqref{inverseLegendre} can be rewritten as
    \[\Lambda_R(V)=\sup_{\sigma\in \mathcal{P}(Y)} (\langle V,\sigma\rangle-I_R(\sigma))\quad\text{for } V\in C_b(X_R).\]
    Notice that $\mathcal{P}(Y)$ is compact\footnote{It is well-known that since $Y$ is compact metric space, $\mathcal{P}(Y)$ is compact in the weak-* topology. The natural inclusion $\mathcal{P}(Y)\hookrightarrow \mathcal{P}(X_R)$ is continuous; see Appendix~\ref{subsec_McShane}.}, which guarantees at least one $\sigma\in \mathcal{P}(Y)$ satisfying
    \[\Lambda_R(V)=\langle V,\sigma\rangle-I_R(\sigma).\]
    Any probability measure $\sigma\in\mathcal{P}(X_R)$ that satisfies this relation is called an \textit{equilibrium state} for $V$. When the equilibrium state is unique, it is denoted by $\sigma_V$.
    \vspace{3mm}

    \begin{definition}\label{def_mathcalV}
        Let  $\mathcal{V}_R\subset L_b(X_R)$ be the family consisting of all $V\in L_b(X_R)$, such that the following two properties hold:
        \begin{enumerate}
            \item[(1)] In the definition of $\Lambda_R(V)$, the right-hand side of \eqref{pressuredefinition} is convergent, i.e.
            \begin{equation*}
                \Lambda_R(V)=\lim_{\theta\in\Theta_R}\frac{1}{r_\theta}\log \mathbb{E} e^{r_\theta\langle V,L_{\theta}\rangle}.
            \end{equation*}

            \item[(2)] $V$ admits a unique equilibrium state $\sigma_V$.
        \end{enumerate}
        
        We further define $\mathcal{W}_R:=\{\sigma_V: V\in \mathcal{V}_R\}\subset \mathcal{P}(X_R)\subset \mathcal{P}(\mathcal{X})$ and $\mathcal{W}\subset \mathcal{P}(\mathcal{X})$ by
        \[\mathcal{W}=\bigcup_{R>0} \mathcal{W}_R.\]
    \end{definition}
    \vspace{3mm}
    
    The following theorem concerning the level-2 local LDP is the main result for our RDS model. Let $C_+(Y)$ consist of all positive continuous functions on the compact space $Y$. 

    \begin{theorem}\label{Thm LDP}
        Suppose that hypotheses \hyperlink{AC}{$\color{black}\mathbf{(AC)}$}, \hyperlink{I}{$\color{black}\mathbf{(I)}$} and \hyperlink{C}{$\color{black}\mathbf{(C)}$} are satisfied, and that there exists $\zeta\in \mathcal E$ such that $S(z,\zeta)=z$. Then there exists a good rate function  $\mathcal{I}\colon \mathcal{P}(\mathcal{X})\to [0,\infty]$ given by
        \begin{equation}\label{I-variantion formula}
            \mathcal{I}(\sigma)=\begin{cases}
                -\inf\limits_{f\in C_+(Y)}\displaystyle\int_Y \log\left(\frac{Pf}{f}\right)\, d\sigma\quad&\text{for }\sigma\in\mathcal{P}(Y),\\
                \qquad\infty\quad&\text{otherwise},
            \end{cases}
        \end{equation}
        such that 
        for any closed subset $F\subset\mathcal{P}(\mathcal{X})$ and open subset $G\subset\mathcal{P}(\mathcal{X})$,
        \begin{align}
            \limsup_{n\to \infty} \frac{1}{n}\log \mathbb{P}(L_{n,x}\in F)&\le -\inf\{\mathcal{I}(\sigma):\sigma\in F\},\label{level2XR_upperbound}\\
            \liminf_{n\to\infty}\frac{1}{n}\log \mathbb{P}(L_{n,x}\in G)&\ge -\inf\{\mathcal{I}(\sigma):\sigma\in G\cap \mathcal{W}\}.\label{level2XR_lowerbound}
        \end{align}

        Moreover, on each bounded invariant set $X_R$, the following holds:

        \begin{enumerate}
            \item[(a)] \hypertarget{a}{{\color{black}\rm\bf Rate function.}} The restriction of $\mathcal{I}$ to $\mathcal{P}(X_R)$ is equal to $I_R$, i.e.~$\mathcal{I}|_{\mathcal{P}(X_R)}=I_R$.

            \item[(b)]  \hypertarget{b}{{\color{black}\rm\bf Uniform local LDP.}} For any $x\in X_R$, \eqref{level2XR_upperbound} holds for any closed set $F\subset \mathcal{P}(X_R)$, and \eqref{level2XR_lowerbound} holds for any open set $G\subset \mathcal{P}(X_R)$ with $\mathcal{W}$ replaced by $\mathcal{W}_R$. In addition, the convergences in the inequalities are uniform with respect to $x\in X_R$.

            \item[(c)]  \hypertarget{c}{{\color{black}\rm\bf Non-triviality.}}   For any $L>0$, there exists $\delta=\delta(L)>0$ such that $\mathcal{V}_{R,L,\delta}\subset \mathcal{V}_R$, where
            \[\mathcal{V}_{R,L,\delta}:=\{V\in L_b(X_R):\Lip(V)<L,\ \Osc(V)<\delta\}.\]
            Furthermore, $\mathcal{W}_R$ is independent of $R$, and thus $\mathcal{W}=\mathcal{W}_R$ for any $R>0$.
        \end{enumerate}
    \end{theorem}

    The proof of Theorem~\ref{Thm LDP} will be carried out in the next section. An immediate application of the aforementioned level-2 local LDP is the corresponding level-1 local LDP.

    \begin{corollary}\label{Cor level1LDP}
        Under the assumptions of Theorem~{\rm\ref{Thm LDP}}, given any test function $f\in L_b(\mathcal{X})$ with $f(z)\not =\langle f,\mu_*\rangle$, there exists a good rate function $\mathcal{I}^f:\mathbb{R}\to [0,\infty]$ and $\varepsilon>0$, such that for any $x\in \mathcal{X}$ and Borel set $B\subset [\langle f,\mu_*\rangle-\varepsilon,\langle f,\mu_*\rangle+\varepsilon]$ with $B\subset \overline{B^\circ}$,
        \begin{equation}
            \lim_{n\to \infty} \frac{1}{n}\log \mathbb{P}_x \left(\frac{1}{n}(f(x_1)+\cdots +f(x_n))\in B\right)=-\inf \{\mathcal{I}^f(p):p\in B\}.\label{RDS_level1}
        \end{equation}
        Moreover, the convergence in \eqref{RDS_level1} holds locally uniform with respect to $x\in \mathcal{X}$.
    \end{corollary}
    
    The proof of Corollary~\ref{Cor level1LDP} is based on a modification of the contraction principle (see \cite[Theorem 4.2.1]{DZ10}), and we leave this technical proof to the appendix; see Appendix~\ref{sec_prooflevel1}.

    \begin{remark}
        According to the proof of Corollary~{\rm\ref{Cor level1LDP}}, for any $p\in [\langle f,\mu_*\rangle-\varepsilon,\langle f,\mu_*\rangle+\varepsilon]$, there exists $\sigma_p\in \mathcal{W}$ such that $\langle f,\sigma_p\rangle=p$ and $\mathcal{I}(\sigma_p)=\mathcal{I}^f(p)$. In particular, $\mathcal{I}^f(p)<\infty$. If the state $z$ in hypothesis \hyperlink{I}{$\color{black}\mathbf{(I)}$} satisfies that $\mathcal{A}(\{z\})\not =\{z\}$, then one may find $f\in L_b(\mathcal{X})$ with $f(z)\not=\langle f,\mu_*\rangle$. Therefore, in this case, the family $\mathcal{W}$ is uncountable.
    \end{remark}
    
    \begin{remark}
        We end this subsection with some further comments on our approach.

        \begin{enumerate}
            \item[(1)] This RDS setting contributes to novel results and improvements on the large deviations for random PDEs. In the context of wave equations, we establish the local LDP for arbitrary initial data belonging to the energy space. Regarding the Navier–Stokes system, we demonstrate the LDP {\rm(}without localness{\rm)} when the viscosity is sufficiently large.
            
            \item[(2)] The localness in Theorem~{\rm\ref{Thm LDP}} is indispensable, as suggested by Example~{\rm \ref{toymodel}}. It is noteworthy that {\rm \cite[Example 1]{BD96}} displays a Markov chain that is mixing, whereas the LDP can only be local. If $Y=\supp(\mu_*)$ and the LDP holds for initial states therein, we can recover the full strength of the LDP.
        \end{enumerate}
    \end{remark}

    \subsection{Proof of the Main Theorem}\label{subsec_maintheorem}
    
Making use of the RDS approach for LDP, we are now in a position to establish the local LDP for random wave equations  (\ref{waveequation}), as indicated in Figure~\ref{fig1}. Indeed, the \hyperlink{MT}{\color{black}Main Theorem} is a direct consequence of Theorem~\ref{Thm LDP} and Corollary~\ref{Cor level1LDP}, while the verification of hypotheses \hyperlink{AC}{$\color{black}\mathbf{(AC)}$}, \hyperlink{I}{$\color{black}\mathbf{(I)}$} and \hyperlink{C}{$\color{black}\mathbf{(C)}$} contributes to the main content of the proof.  More precisely, this has been achieved by the 
following technical route established in \cite{LWXZZ24}:

\begin{itemize}

\item [(1)]``Global stability of the unforced wave equation in Proposition~\ref{prop_globalstability}"  implies  hypothesis~\hyperlink{I}{$\color{black}\mathbf{(I)}$}, see Section~\ref{Irreducibility};

\item [(2)] ``Asymptotic compactness for the perturbed system in Proposition~\ref{prop_ac}" implies  hypothesis~\hyperlink{AC}{$\color{black}\mathbf{(AC)}$}, see Section~\ref{Asymptotic compactness};

\item [(3)]``A stabilization property from the control theory in Proposition~\ref{prop_control}" implies hypothesis~\hyperlink{C}{$\color{black}\mathbf{(C)}$}, see Section~\ref{Coupling condition}.
\end{itemize}

The rigorous statements and detailed proofs are collected in \cite{LWXZZ24}. For the reader's convenience,  we briefly sketch the core ideas for Propositions~\ref{prop_globalstability}-\ref{prop_control} in the following.

    \vspace{3mm}
    Let us begin with a brief summary of the RDS setting for the wave equation \eqref{waveequation}. When the random external force $\eta$ in equation \eqref{waveequation} is replaced by a deterministic one, the global well-posedness for the equation is known. More precisely,  for any $\zeta\in L^2(D_T)$ and $\boldsymbol{u}=(u_0,u_1)\in \mathcal{H}$, there exists a unique strong solution $u[t]$ up to time $T$, such that $u[t]\in \mathcal{H}$ for any $t\in [0,T]$. This leads to a locally Lipschitz mapping defined as
    \[S\colon \mathcal{H}\times L^2(D_T)\to \mathcal H,\quad S(\boldsymbol{u},\zeta)=u[T].\]
    
    Meanwhile, recall the random noise $\eta$ is in the form of \eqref{noise-structure1}, which implies that $(\eta_n)_{n\in\mathbb{N}_0}$ are i.i.d.~$L^2(D_T)$-valued random variables. Moreover, by~\eqref{noise-structure2}, the support of the common law of $\eta_n$ is a compact subset of $L^2(D_T)$.   Under the above settings, the solution of \eqref{waveequation} at times $nT$, constitute an RDS in the context of \eqref{RDS}, where $S$ is restricted on $\mathcal{H}\times \supp(\mathscr{D}(\eta_0))$. Specifically,
    \begin{equation}
        \boldsymbol{u}_n=S(\boldsymbol{u}_{n-1},\eta_{n-1}),\quad \boldsymbol{u}_0=(u_0,u_1)\in \mathcal{H}.\label{wave-RDS}
    \end{equation}

    \vspace{3mm}
    With theses preparations, we are able to finish the proof of the \hyperlink{MT}{\color{black}Main Theorem}.
    \begin{proof}[Proof of the \hyperlink{MT}{\color{black}Main Theorem}]
    
    As demonstrated in \cite{LWXZZ24}, under the setting for the damping term $a(x)$ and cutoff function $\chi(x)$, there exists an intrinsic quantity $\mathbf{T}>0$ with the following properties: for any $T>\mathbf{T}$ and $B_0>0$, there exists an integer $N=N(T,B_0)$ such that if 
    $b_{jk}>0$ for $j,k=1,\cdots,N$, then hypotheses \hyperlink{AC}{$\color{black}\mathbf{(AC)}$}, \hyperlink{I}{$\color{black}\mathbf{(I)}$} and \hyperlink{C}{$\color{black}\mathbf{(C)}$} are satisfied. A more detailed reasoning is contained in the rest of this section, based on Propositions~\ref{prop_globalstability}-\ref{prop_control} below. Thus Theorem~\ref{Thm LDP} and Corollary~\ref{Cor level1LDP} are applicable, and the Main Theorem follows as a consequence.
    \end{proof}

    \vspace{3mm}

    Let us also mention that the level-2 local LDP for the wave equation is actually implied by Theorem~\ref{Thm LDP}, whose explicit statements are omitted for the sake of simplicity. The rest of this section is devoted to recalling a series of auxiliary results for verifying the assumptions for random wave equation \eqref{waveequation}, followed by discussions on the main ideas for each conclusion. 

    \subsubsection{Verification of irreducibility}\label{Irreducibility} 
    
    As mentioned above, the irreducibility follows from the global stability for unforced problem, i.e.~(\ref{waveequation}) with zero external forces. To this end, we shall use the following result due to Zuazua \cite{Zuazua-90}. Recall $E(\cdot)$ refers to the energy \eqref{energy}.
    \begin{proposition}\label{prop_globalstability}
       Under the assumptions of the \hyperlink{MT}{\color{black}Main Theorem}, there exist constants $C,\beta>0$ such that for any $\boldsymbol{u}\in \mathcal{H}$,
        \begin{equation}
            E(S(\boldsymbol{u},0))\le Ce^{-\beta T} E(\boldsymbol{u}).\label{globalstability}
        \end{equation}
    \end{proposition}

    The localized damping contributes to great difficulty to the demonstration of \eqref{globalstability}. When the damping coefficient $a(x)$ is a positive constant, this can be derived via a standard energy method. Various achievements has been made in recent decades in the case of localized damping. 
    It is noteworthy that such type of questions are also investigated for other dispersive equations; see, e.g.~\cite{DLZ-03,Laurent-10} for Schr\"odinger equations, and more recently, \cite{KX-22,CKX-23} for geometric wave equations.
    
    There are two principal methods in this direction. One is the multiplier method in the case that the damping satisfies the $\Gamma$-condition as in our setting; see \cite{Zuazua-90} for details. The other one is based on microlocal analysis, and imposes yet a weaker geometric condition on the damping coefficient, called geometric control condition (GCC for abbreviation); see, e.g.~\cite{BLR92}. Roughly speaking, the damping effect propagates along generalized geodesics, and the GCC requires any geodesic of length $L_0>0$ to meet the effective domain of the damping, i.e.~where $a(x)>0$. The GCC is almost also a necessary condition for exponential decay, and in our setting leads to a result similar to \eqref{globalstability} (see \cite{JL-13}), while the rate $\beta$ in that paper depends on the scale of $u[0]$. It is hopeful that this dependence can be removed. Another reason we impose the $\Gamma$-condition rather than the GCC lies in the derivation of the asymptotic compactness.
       
       \begin{proof}[Sketched proof of irreducibility]
       Thanks to Proposition~\ref{prop_globalstability} and the fact that $0\in \supp(\mathscr{D}(\eta_0))$ (ensured by the noise structure), hypothesis \hyperlink{I}{$\color{black}\mathbf{(I)}$} can be verified for $z=(0,0)$. In addition, $z$ is also a fixed point for the RDS \eqref{wave-RDS}, as $S(z,0)=z$.
    \end{proof}     

    \subsubsection{Verification of asymptotic compactness}\label{Asymptotic compactness}

        The following proposition regarding the so-called $\mathcal{H}$-$\mathcal{H}^{\scriptscriptstyle 4/7}$ asymptotic compactness is established in \cite[Theorem 1.1]{LWXZZ24}. Here the space $\mathcal{H}^{\scriptscriptstyle 4/7}$ is defined as follows: denote by $H^s (s>0)$ the domain of fractional power $(-\Delta)^{s/2}$, and set $\mathcal H^s=H^{1+s}\times H^s$. Then $\mathcal{H}^{\scriptscriptstyle 4/7}$
        is compactly embedded in $\mathcal{H}$.
        
        \begin{proposition}\label{prop_ac}
             Under the assumptions of the \hyperlink{MT}{\color{black}Main Theorem}, there exists a bounded set $Y_0\subset \mathcal{H}^{\scriptscriptstyle 4/7}$ and constants $C,\kappa>0$, such that for any $\boldsymbol{u}\in \mathcal{H}$,
            \[\dist_{\mathcal{H}}(\boldsymbol{u}_n,Y_0)\le C(1+E(\boldsymbol{u}))e^{-\kappa n}.\]
        \end{proposition}

        Let us briefly outline the proof of Proposition~\ref{prop_ac}. Multiplying the equation \eqref{waveequation} by
        \[(x-x_0)\cdot \nabla u\]
        and integrating over $D_T$, one can derive a discrete-time monotonicity-type relation: there exists $A>0$ and $T>T_0$, such that if $E(u[0])>A$, then $E(u[T])<E(u[0])/2$. This is natural in view of the global stability \eqref{globalstability} when the external force vanishes, and that when $E(u[0])>A$ the noise is a small term in comparison. As a result, we find that there is a bounded absorbing set in $\mathcal{H}$: for $n$ sufficiently large, we must have $\|\boldsymbol{u}_n\|_{\mathcal{H}}<CA$ for some $C=C(T,B_0)$. Similarly, if the initial data belongs to $\mathcal{H}^{\scriptscriptstyle 4/7}$, then there exists a bounded absorbing set in $\mathcal{H}^{\scriptscriptstyle 4/7}$. 
        
        Now given $\boldsymbol{u}\in \mathcal{H}$, we can decompose the solution into $u[t]=v[t]+w[t]$, where $v$ and $w$ are solutions to the equations
        \[\begin{cases}
            \boxempty v-a(x)\partial_t v=0,\\
            v|_{\partial D}=0,\\
            (v,\partial_t v)(0)=(u_0,u_1),
        \end{cases}\quad \text{and}\quad \begin{cases}
            \boxempty w-a(x)\partial_t w=\eta-u^3,\\
            w|_{\partial D}=0,\\
            (w,\partial_t w)(0)=(0,0),
        \end{cases}\]
        respectively. Then $v[t]$ decays exponentially in $\mathcal{H}$ by \eqref{globalstability}, and $w[t]$ possesses additional regularity in $\mathcal{H}^{\scriptscriptstyle 4/7}$ and hence is absorbed into a bounded set $Y_0$ in $\mathcal{H}^{\scriptscriptstyle 4/7}$.

     \begin{proof}[Sketched proof of asymptotic compactness]

     Employing Proposition~\ref{prop_ac}, one can check that   $Y:=\mathcal{A}(Y_0)$ is a compact invariant set for the RDS \eqref{wave-RDS}, see, e.g.~\cite[Proposition 2.2]{LWXZZ24}. Moreover, hypothesis \hyperlink{AC}{$\color{black}\mathbf{(AC)}$} follows for such $Y$, together with $V(\boldsymbol{u})=C(1+E(\boldsymbol{u}))$.
    \end{proof}

    \subsubsection{Verification of coupling condition}\label{Coupling condition} 
   
    Inspired by the previous works \cite{Shi-15,Shi-21}, the verification of coupling condition can be reduced to a stabilization property for the associated controlled system. The following result is demonstrated in \cite[Theorem 1.3]{LWXZZ24}. 
    \begin{proposition}\label{prop_control}
             Under the assumptions of the \hyperlink{MT}{\color{black}Main Theorem}, there exist constants $N\in\mathbb{N}$ and $C,d>0$, such that for any $\boldsymbol{u},\hat{\boldsymbol{u}}\in Y$ with $\|\boldsymbol{u}-\hat{\boldsymbol{u}}\|_{\mathcal{H}}\leq d$, and $h\in \supp(\mathscr{D}(\eta_0))$, there exists a control $\zeta=\zeta(\boldsymbol{u},\hat{\boldsymbol{u}},h) \in L^2(D_T)$ satisfying
        \begin{equation}\|S(\boldsymbol{u},h)-S(\hat{\boldsymbol{u}},h+\chi \mathscr{P}_N \zeta)\|_{\mathcal{H}}\le \|\boldsymbol{u}-\boldsymbol{\hat{u}}\|_{\mathcal{H}}/4. \label{coupling-condition1}
        \end{equation}
        Here $\mathscr{P}_N$ is the orthogonal projection from $L^2(D_T)$ to the subspace spanned by $\alpha_k^{\scriptscriptstyle T}(t)e_j(x)(j,k=1,\dots ,N)$. Moreover, the control can be estimated by
        \begin{equation}
            \|\zeta\|_{L^2(D_T)}\le C\|\boldsymbol{u}-\hat{\boldsymbol{u}}\|_{\mathcal{H}}.\label{coupling-condition2}
        \end{equation}
        \end{proposition}

        The demonstration of the proposition on stabilization is a main challenge,  involving frequency analysis, unique continuation, observable inequalities, Carleman estimates and Hilbert uniqueness method. The degeneracy of the control, both in space and frequency, causes extra difficulties to the proof. A moral in \cite{DL-09} is that, “the energy of each scale of the control force depends (almost) only on the energy of the same scale in the states that one wants to control”. One exploits the control to treat the low-frequency terms, and leave the high-frequency terms to decay along the evolution. The idea is clear, whereas the realization of this strategy is extremely intricate. Such frequency analysis strategy has been exploited in earlier works \cite{ADS-16,KX-23,Xiang-23,Xiang-24}.
    
        \begin{proof}[Sketched proof of coupling condition] 
        Hypothesis \hyperlink{C}{$\color{black}\mathbf{(C)}$} follows from combining Proposition~\ref{prop_control} with the optimal coupling methods, see, e.g. \cite{Shi-15,LWXZZ24}. We refer the reader to \cite[Section 6.3]{LWXZZ24} for a rigorous proof.

        Intuitively, writing $\hat{\eta}_0=\eta_0+\chi\mathscr{P}_N\zeta(\boldsymbol{u},\hat{\boldsymbol{u}},\eta_0)$, by~\eqref{coupling-condition2}, for $\boldsymbol{u}$ and $\hat{\boldsymbol{u}}$ sufficiently close, $\hat{\eta}_0$ can be viewed of a small perturbation of $\eta_0$. Moreover, recall that the coefficients in the noise structure \eqref{noise-structure2} satisfy that $b_{jk}>0$ for $j,k=1,\dots ,N$, and the random variables $\theta_{jk}^{n}$ admit $C^1$ densities $\rho_{jk}$. Under such settings, it would imply that
        \begin{equation}
            \|\mathscr{D}(S(\hat{\boldsymbol{u}},\eta_0))-\mathscr{D}(S(\hat{\boldsymbol{u}},\hat{\eta_0}))\|_{\rm TV}\leq C\|\boldsymbol{u}-\hat{\boldsymbol{u}}\|_{\mathcal{H}}\label{coupling-condition3}
        \end{equation}
        for a universal constant $C>0$. Here $\|\cdot\|_{\rm TV}$ denotes the total variation norm of measures, see e.g.~\cite[Section 1.2]{KS-12}.

        Moreover, by \eqref{coupling-condition1}, performing the optimal coupling techniques, there exists a pair of $\mathcal{H}$-valued random variables $(R(\boldsymbol{u},\hat{\boldsymbol{u}}),R'(\boldsymbol{u},\hat{\boldsymbol{u}}))$ such that
        \begin{equation}
            \mathbb{P}(\|R(\boldsymbol{u},\hat{\boldsymbol{u}})-R'(\boldsymbol{u},\hat{\boldsymbol{u}})\|_{\mathcal{H}}>\|\boldsymbol{u}-\hat{\boldsymbol{u}}\|_{\mathcal{H}}/2)\leq 2\|\mathscr{D}(S(\hat{\boldsymbol{u}},\eta_0))-\mathscr{D}(S(\hat{\boldsymbol{u}},\hat{\eta_0}))\|_{\rm TV}.\label{coupling-condition4}
        \end{equation}

        Therefore, hypothesis \hyperlink{C}{$\color{black}\mathbf{(C)}$} follows from~\eqref{coupling-condition3} and \eqref{coupling-condition4} with $q=1/2$ and $g(r)=2Cr$.
        \end{proof}

    \subsection{Two more applications}\label{subsec_furthercomments}
    
    We present two more applications of our RDS approach, and comments on the ``local" constraint. The first is a toy model illustrating that the localness can be intrinsic in general. The second is 2D Navier--Stokes system with large viscosity, which indicates that it is possible to remove the localness when $Y=\mathcal{A}(=\supp(\mu_*))$. That is to say, the attracting compactum $Y$ is exactly the attainable set $\mathcal{A}$.

    \subsubsection{An elementary Markov chain}

    We provide the following example to show the localness is inherent in our setting. In such a simple case, the LDP can be computed directly.
    \begin{example}\label{toymodel}
        Let $(x_n)_{n\in \mathbb{N}_0}$ be a Markov chain on $\mathcal{X}:=\{0,1,2\}$ with initial condition $x_0=x$ and a non-irreducible transition probability defined as
        \begin{equation*}
            \pi(0,0)=\pi(2,1)=1,\quad \pi(1,0)=\pi(1,1)=1/2.
        \end{equation*}
        The unique invariant measure is $\delta_0$, and this RDS is exponential mixing:
        \begin{equation*}
            \|\mathscr{D}(x_n)-\delta_0\|_L^* \le 2^{-(n-1)} \quad \text{for any } x\in \mathcal{X}\text{ and }n\in\mathbb{N}.
        \end{equation*}

       By construction, this example fits into our RDS model with $z=0$, $\mathcal{A}=\{0\}$ and $Y=\{0,1\}$. The corresponding empirical measures $L_{n,x}(n\ge 1)$ are given by
        \begin{equation*}
            {L}_{n,x}=\begin{cases}
                \delta_0\quad &\text{for }x=0,\\
                (1-2^{-n})\delta_0+2^{-n}\delta_1\quad &\text{for }x=1,\\
               (1-2^{-(n-1)})\delta_0+2^{-(n-1)}\delta_1\quad &\text{for }x=2.
            \end{cases}
        \end{equation*}
        Thus for each $x\in \mathcal{X}$, the sequence $(L_{n,x})_{n\in \mathbb{N}}$ satisfies the LDP with good rate function $I_x$, where
        \begin{equation*}
            I_0(\sigma)=\begin{cases}
                0\quad &\text{for }\sigma=\delta_0,\\
                \infty &\text{otherwise},
            \end{cases}\quad \text{and}\quad
            I_1(\sigma)=I_2(\sigma) =\begin{cases}
                t\log 2\quad &\text{for }\sigma=t\delta_0+(1-t)\delta_1,\\
                \infty &\text{otherwise}.
                \end{cases}
        \end{equation*}

        We provide some comments on this example:        
        \begin{enumerate}
            \item[(1)] Noting that $I_0\neq I_1$, a uniform LDP cannot exist without localness. Intuitively, this defection can be attributed to that the state $1$ is not reachable from $0$, i.e.~$Y\not =\mathcal{A}$. In this example, if we consider the Markov process issued from each initial state $x\in \mathcal{X}$, then the LDP holds with a rate function $I_x$ depending on $x$. It would be more accurate if we consider the LDP with respect to each initial value, but yet more complicated.

            \item[(2)] For $V\in L(\mathcal{X})$, direct computations yield that $V\not\in\mathcal{V}$ whenever $V(1)-V(0)>\log 2$, where $\mathcal{V}$ is as in Definition~\ref{def_mathcalV}.  This reflects that the constraint $\Osc(V)<\delta$ for $V\in \mathcal{V}$ is necessary. Intuitively, when $V(1)-V(0)$ is sufficiently large, although the probability of $x_n=1$ is exponentially small, it still dominates other terms when multiplied by $e^{nV(1)}$. 
            
            \item[(3)] Whether $V\in \mathcal{V}$ or not only depends on the restriction of $V$ to $Y$. In this example, the conclusion follows from the fact that $x_1\in Y$ regardless of the initial condition. Consequently, the long-time behaviour is the same as that of the system initiating in $Y$. Meanwhile, in our general setting, although the process \eqref{RDS} does not necessarily enter the compactum $Y$ in hypothesis $(\mathbf{AC})$, this property remains valid; see Corollary~\ref{VrestricY}.
            
        \end{enumerate}
    \end{example}

    \subsubsection{2D Navier--Stokes system}\label{subsubsec_NS}
   
   We further consider a 2D incompressible Navier--Stokes system with space-time localized noise, which is proved to be exponential mixing by Shirikyan \cite{Shi-15}.  The system under consideration reads
    \begin{equation}
        \begin{cases}
            \partial_t u+u\cdot \nabla u-\nu \Delta u+\nabla p=\eta(t,x),\quad\dvg u=0,\quad x\in D,\\
            u|_{\partial D}=0,\\
            u(0)=u_0,
        \end{cases}\label{NSsystem}
    \end{equation}
    where $D$ is a smooth bounded domain in $\mathbb{R}^2$, the constant $\nu>0$ is the viscosity, $u=(u^1,u^2)$ is the velocity field and $p$ is the pressure. The phase space for equation~\eqref{NSsystem} is specified as 
    \[H:=\{u\in L^2(D;\mathbb{R}^2): \dvg u=0\text{ in }D,\ u\cdot n=0\text{ on }\partial D\}.\]
    
    \vspace{1mm}    
    The localized noise structure is specified as follows. Fix any open set $Q\subset (0,1)\times D$ and let $(\varphi_j)_{j\in \mathbb{N}}\subset H^1(Q;\mathbb{R}^2)$ be an orthonormal basis in $L^2(Q;\mathbb{R}^2)$.     
    \begin{enumerate}
    \item[$\mathbf{(S1')}$] ({\bf Localized structure}) {\it Let $\chi(\cdot,\cdot)\in C^\infty_0(Q)$ be a non-zero function.
    }
    \end{enumerate}

    \begin{enumerate}
    \item[$\mathbf{(S2')}$]  
    {\it 
    Let $(\rho_{j})_{j\in \mathbb{N}}$ be a sequence of probability density functions supported by $[-1,1]$, which is $C^1$ and satisfies  $\rho_{j}(0)>0$. }
    \end{enumerate}

    \vspace{1mm}
    The noise $\eta(t,x)$ in \eqref{NSsystem} is of the form
   \begin{equation*}
    \begin{aligned}
    &\eta(t,x)= \eta_n(t-n,x)\quad \text{for } t\in (n,n+1)\text{ and }n\in\mathbb{N}_0,\\ 
    &\displaystyle\eta_n(t,x)=\chi(t,x)\sum_{j\in\mathbb{N}}b_{j}\theta^{n}_{j}\varphi_j(t,x)\quad \text{for }t\in (0,1),
    \end{aligned}
    \end{equation*}
    where $\theta_{j}^{n}$ are independent random variables such that $(\theta_{j}^{n})_{n\in\mathbb{N}_0}$ has a common density $\rho_{j}$, and  $(b_j)_{j\in\mathbb{N}}$ is a sequence of non-negative numbers. By construction, $\eta_n$ are i.i.d.~$L^2(Q)$-valued random variables.

    As the well-posedness of 2D Navier-Stokes system is well-known, the corresponding time-1 resolvent operator $S$ for \eqref{NSsystem} is well-defined and locally Lipschitz. That is,
    \begin{equation*}
        S\colon H\times \supp(\mathscr{D}(\eta_0))\to H,\quad S(u(0),\eta_0)=u(1).
    \end{equation*}
    Then equation \eqref{NSsystem} defines an RDS by the relations
    \begin{equation}
        u_n=S(u_{n-1},\eta_{n-1})=u(n),\quad u_0=u.\label{NS-RDS}
    \end{equation}
    
    \vspace{3mm}
    With the above settings, another application of Theorem~\ref{Thm LDP} for the Navier-Stokes equation \eqref{NSsystem}, with the localness in~\eqref{level2XR_lowerbound} removed, is contained in the following. 
    
     \begin{proposition}\label{NSlevel2}

        Under the settings $\mathbf{(S1')}$ and $\mathbf{(S2')}$, for any ${B}_0>0$, there exists a constant  $\nu_0=\nu_0(B_0)>0$ such that if 
    \begin{equation*}
    \nu\ge \nu_0,\quad \sum_{j\in\mathbb{N}}b_{j}\|\varphi_j\|_{H^1(Q)} \leq {B}_0,\quad\text{and}\quad  b_{j}\neq0\ \text{ for } j\in\mathbb{N},
    \end{equation*}
   then there exists a good rate function $\mathcal{I}\colon\mathcal{P}(H)\rightarrow[0,\infty]$ satisfying $\mathcal{I}(\sigma)=\infty$ for any $\sigma\notin\mathcal{P}(\mathcal{A}(\{0\}))$, such that for any closed subset $F\subset\mathcal{P}(H)$ and open subset $G\subset\mathcal{P}(H)$,
        \begin{align*}
            \limsup_{n\rightarrow\infty} \frac{1}{n}\log \mathbb{P}_u\left(\frac{1}{n}(\delta_{u_1}+\cdots +\delta_{u_n})\in F\right)&\le -\inf\{\mathcal{I}(\sigma):\sigma\in F\},\\       
            \liminf_{n\rightarrow\infty}\frac{1}{n}\log \mathbb{P}_u\left(\frac{1}{n}(\delta_{u_1}+\cdots +\delta_{u_n})\in G\right)&\ge -\inf\{\mathcal{I}(\sigma):\sigma\in G\}.
        \end{align*}
        Moreover, the above convergences hold locally uniformly with respect to $u\in H$.
    \end{proposition}   

    \begin{proof}

    Applying our abstract results, it suffices to verify the assumptions of Theorem~\ref{Thm LDP}, which implies the level-2 local LDP, and then to remove the localness restriction. In fact, when $\nu$ is sufficiently large, we can show that hypothesis $(\mathbf{AC})$ holds with $Y=\mathcal{A}(\{0\})$. As a result, we have $\mathcal{I}(\sigma)=\infty$ for any $\sigma\not \in \mathcal{P}(\mathcal{A}(\{0\}))$ by \eqref{D(I)subsetP(Y)}.

    By \cite[Proposition 2.1.21]{KS-12}, for any $\nu\ge 1$, there exist universal constants $\beta_1\in (0,1)$ and $C_1>0$ such that
        \begin{equation}
            \|S(u,\zeta)\|_{H}\le \beta_1\|u\|_{H}+C_1\|\zeta\|_{L^2(Q)}\quad\text{for }u\in H\text{ and }\zeta\in L^2(Q).\label{NSdisspation}
        \end{equation}
        Therefore, there exists $R=R(B_0)$ such that $B_H(R)$ is invariant for the RDS~\eqref{NS-RDS}. Moreover, one can find $n_0=n_0(\|u_0\|_{L^2(D)})\in \mathbb{N}$, such that $u_n\in B_H(R)$ for any $n\ge n_0$. The smoothing effect of the parabolic system implies that the image of $B_H(R)$, denoted by
        \[X:=S(B_H(R)\times \supp(\mathscr{D}(\eta_0)))\]
        is a compact invariant subset of $H$. Thus we may reduce the phase space to $X$ and only consider initial data $u\in X$. As $X$ is compact (and hence bounded), there is no need to further concerning each bounded invariant set $X_R$. We can define $\Lambda$, $I$, $\mathcal{V}$, $\mathcal{W}$ on $X$ as in Section~\ref{subsec_mixing}.
         
    \begin{itemize}
        \item [(1)] {\it Asymptotic compactness.} By \cite[Proposition 5.9]{MN-20}, there exists $\nu_0=\nu_0(B_0)>0$ such that for $\nu\ge \nu_0$, there exists $\beta_2\in (0,1)$ such that 
        \begin{equation}\label{NS-AC2}
        \|S(u,\zeta)-S(u',\zeta
        )\|_H\le \beta_2\|u-u'\|_H\quad \text{for }u,u'\in X\text{ and }\zeta\in \supp(\mathscr{D}(\eta_0)).
        \end{equation}
        In particular, iterating~\eqref{NS-AC2} yields that for any $u\in X$, almost surely
        \begin{equation*}
            \dist_{H}(u_n,\mathcal{A}(\{0\}))\leq  \beta_2^{n}\|u\|_{H}.
        \end{equation*}
       This completes the verification of hypothesis \hyperlink{AC}{$\color{black}\mathbf{(AC)}$} with $Y=\mathcal{A}:=\mathcal{A}(\{0\})$.

        \item [(2)] {\it Irreducibility.} As in the verification of wave system, the dissipation \eqref{NSdisspation} and $0\in \supp(\mathscr{D}(\eta_0))$ imply hypothesis \hyperlink{I}{$\color{black}\mathbf{(I)}$} with $z=0$. Moreover, $z$ is also a fixed point of the RDS \eqref{NS-RDS}, i.e.~$S(0,0)=0$. See \cite[Section 4.3]{Shi-15} for a rigorous proof.

        \item [(3)] {\it Coupling condition.}  Hypothesis \hyperlink{C}{$\color{black}\mathbf{(C)}$} is a direct consequence of \cite[Proposition 2.6]{Shi-15} with $q=1/2$ and $g(r)=Cr$ in the present setting.
        \end{itemize}

        \begin{itemize}
        \item [(4)] {\it Removing localness. } According to the original Kifer's criterion \cite[Theorem 2.1]{Kifer-90} (see also Remark~\ref{Rmk_Kifer}), it suffices to prove that $L(X)\subset \mathcal{V}$. To this end, we only need to combine the following two observations:
        \begin{itemize}
            \item [(i)] The level-2 LDP for equation~\eqref{NSsystem} on the compactum $Y=\mathcal{A}$ is derived in \cite{PX22}, where it is proved that $L(Y)\subset \mathcal{V}(Y)$. See Definition~\ref{def_V(Y)} for  $\mathcal{V}(Y)$.
            
            \item [(ii)] By Corollary~\ref{VrestricY}, for $V\in L(X)$, we have $V\in \mathcal{V}$ if and only if $V|_Y \in \mathcal{V}(Y)$.
        \end{itemize} 
    \end{itemize}

    Collecting (1)-(4), the proof is now complete.   \end{proof}

    We end this section with some comments on Proposition~\ref{NSlevel2}. The corresponding level-1 LDP for any test function $f\in C_b(H)$ immediately follows from the contraction principle.  In addition, from our arguments, one can see that the local level-2 LDP is satisfied for equation~\eqref{NSsystem} with any viscosity $\nu>0$.  It remains open whether the localness can be removed for any $\nu>0$.  We also mention that in \cite[Section 5.5.1]{MN-20}, Martirosyan and Nersesyan prove a similar result for the 2D Navier-Stokes system with large viscosity perturbed a bounded kick force.

    \section{Proof of LDP in the abstract setting}\label{sec_proofofabstractldp}

    In this section we lay out the proof of Theorem~\ref{Thm LDP} for our RDS approach. We first establish a new variant of Kifer's criterion in Section~\ref{subsubsec_Kifer}, based on asymptotic exponential tightness. Then in  Section~\ref{subsec_prooflevel2LDP} we apply this criterion to finish the proof of Theorem~\ref{Thm LDP}. Figure~\ref{fig2} illustrates the interconnections among various ingredients along the sequel.

\setlength{\abovecaptionskip}{0.3cm}

    \begin{figure}[H]
\centering
\begin{tikzpicture}[node distance=2cm]
\footnotesize {
\node (AET) [idea] {\hyperlink{AET}{$\color{black}\mathbf{(AET)}$}};
\node (localKifer) [process, right of=AET, xshift=2cm] {Kifer's Criterion\\
(Thm.~\ref{Thm Kifer})};
\node (LDP) [process, right of=localKifer, xshift=3cm] { \hyperlink{b}{(b)} Local LDP \\
(Sec.~\ref{subsubsec_(b)})};
\node (Vdelta) [process, below of=LDP] {\hyperlink{c}{(c)} Non-triviality \\
(Sec.~\ref{subsubsec_(c)})};
\node (restriction) [process, below of=localKifer] {Reduction to $Y$\\(Cor.~\ref{VrestricY})};
\node (AC) [idea, left of=restriction, xshift=-2cm] {\hyperlink{AC}{$\color{black}\mathbf{(AC)}$}};
\node (FK) [idea, below of=AC] {\hyperlink{I}{$\color{black}\mathbf{(I)}$} \& \hyperlink{C}{$\color{black}\mathbf{(C)}$}};
\node (VdeltaY) [process, right of=FK, xshift=2cm] {Non-triviality on $Y$\\(Lem.~\ref{nontrivialonY})};
\node (Ientropy) [process, below of=Vdelta] { \hyperlink{a}{(a)} Rate Function\\
(Sec.~\ref{subsubsec_(a)})};
\node (Thm) [above of=LDP, yshift=-1cm] {Theorem~\ref{Thm LDP}};

\node (pt1) [point, right of=VdeltaY, xshift=0.2cm]{};
\node (pt2) [point, right of=restriction, xshift=0.2cm]{};

\node (pt3) [point, below of=VdeltaY, yshift=0.5cm]{};
\node (pt4) [point, below of=Ientropy, yshift=0.5cm]{};

\draw [arrow](AET) -- (localKifer);
\draw [arrow](localKifer) -- (LDP);
\draw [arrow](AC) -- (restriction);
\draw [dasharrow](AC) -- (AET);
\draw [arrow](FK) -- (VdeltaY);
\draw [arrow](restriction) -- (Vdelta);

\draw[thick] (VdeltaY) -- (pt1);
\draw[thick] (pt1) -- (pt2);

\draw[thick] (Ientropy) -- (pt4);

\draw[thick] (pt4) -- (pt3);
\draw[arrow] (pt3) -- (VdeltaY);

\draw[dashed] (6.8,1.5) -- (11.2,1.5);
\draw[dashed] (6.8,1.5) -- (6.8,-5);
\draw[dashed] (11.2,1.5) -- (11.2,-5);
\draw[dashed] (6.8,-5) -- (11.2,-5);

}
\end{tikzpicture}
\caption{}
\label{fig2}
\end{figure}

    Let us overview some key points along the demonstration.

    \begin{itemize}
        \item[\tiny$\bullet$] \textbf{\hyperlink{AET}{\color{black}$(\mathbf{AET})$}-based Kifer's criterion.} Aiming to study the LDP for dynamical systems and Markov processes, Kifer \cite{Kifer-90} puts forward a sufficient criterion for the LDP of random probability measures; see also Remark~\ref{Rmk_Kifer}. We will establish a variant of this result, which leads to the local LDP, i.e.~statement \hyperlink{b}{(b)} of Theorem~\ref{Thm LDP}. \\
        Indeed, for decades researchers have exploited the concept of exponential tightness. However, for hyperbolic cases the lack of compactness represents new challenges. As this condition seems invalid in our setting, we seek for a new sufficient condition, called asymptotic exponential tightness.

        \item[\tiny$\bullet$] \textbf{Reduction to compactum $Y$ via \hyperlink{AC}{\color{black}$(\mathbf{AC})$}.} In the light of asymptotic compactness, we can convert the judgement of $V\in \mathcal{V}_R$ to investigating its restriction to $Y$. This is the first step towards the non-triviality, i.e.~statement \hyperlink{c}{(c)} of Theorem~\ref{Thm LDP}.\\
         In fact, we define $\mathcal{V}(Y)$ as an analogue of $\mathcal{V}_R$. Relying on estimates for the pressure function, we have that $V\in \mathcal{V}_R$ if and only if $V|_Y\in \mathcal{V}(Y)$.

        \item[\tiny$\bullet$] \textbf{Non-triviality on $Y$.} To conclude the non-triviality, in the second step we define 
        \[\mathcal{V}_{L,\delta}(Y):=\{V\in L(Y):\Lip(V)<L,\ \Osc(V)<\delta\},\]
        and it suffices to show $\mathcal{V}_{L,\delta}(Y)\subset \mathcal{V}(Y)$. This issue is closely related to the long-time behaviour of Feynman--Kac semigroups; see Proposition~\ref{Prop Feynman-Kac stability}. Invoking ideas originating from \cite{DV-75}, we can verify statement \hyperlink{a}{(a)} of the theorem and combine it with this proposition to conclude the proof. It is worth noting that Proposition~\ref{Prop Feynman-Kac stability} requires the hypotheses \hyperlink{I}{$\color{black}\mathbf{(I)}$} and \hyperlink{C}{$\color{black}\mathbf{(C)}$}, and follows from multiplicative ergodic theorems introduced in \cite{JNPS-15,MN-20}.
    \end{itemize}

    \subsection{(AET)-based Kifer's criterion}\label{subsubsec_Kifer}

    We present a new variant of Kifer's criterion, based on asymptotic exponential tightness. It is the basis of this paper, and applies to our RDS model to yield the LDP in desire, as illustrated in Figure~\ref{fig1}. In this subsection, we temporarily forget the RDS setting and consider random probability measures.
    
    Let $\mathfrak{X}$ be a Polish space, and $L_\theta\in \mathcal{P}(\mathfrak{X})(\theta\in \Theta)$ a net of random probability measures, where $(\Theta,\prec)$ is a directed set. The scale $r_\theta\in (0,\infty)$ satisfies $\lim_{\theta\in \Theta} r_\theta=\infty$. We make the following hypothesis in compensation for the lack of compactness:
    \begin{itemize}
        \item[\hypertarget{AET}{$\mathbf{(AET)}$}]{\bf Asymptotic exponential tightness.} For any $l>0$, there exists a compact subset $\mathcal{K}_l\subset \mathcal{P}(\mathfrak{X})$, such that whenever $r>0$, we have
        \begin{equation}
            \limsup_{\theta\in \Theta} \frac{1}{r_\theta} \log \mathbb{P}(L_\theta\in B_{\mathcal{P}(\mathfrak{X})}(\mathcal{K}_l,r)^C)\le -l.\label{AET}
        \end{equation}
    \end{itemize}

    \begin{remark}\label{Rmk_AETvsET}
        In the introduction we stated the asymptotic exponential compactness in the context of a sequence of probability measures $\mu_n$. This is in fact compatible with the above definition, except that the sequence is promoted to a net, and $\mu_\theta\in \mathcal{P}(\mathcal{P}(\mathfrak{X}))$ is taken as the law of $L_\theta$. 
        
        It is worth mentioning that  hypothesis \hyperlink{AET}{$\color{black}\mathbf{(AET)}$} cannot imply exponential tightness in general. For example, let $\mathfrak{X}=l^1(\mathbb{N})$, $\Theta=\mathbb{N}\times \mathbb{N}$ with partial order $(n_1,m_1)\prec (n_2,m_2)$ if and only if $n_1\le n_2$, and consider the net $\mu_{n,m}:=\delta_{e_m/n}\in \mathcal{P}(\mathfrak{X})$, where $e_m$ is the $m$-th unit vector. 
    \end{remark}
        
    As in Section~\ref{subsec_ldpstatement}, we define the pressure function $\Lambda\colon C_b(\mathfrak{X})\to \mathbb{R}$ by \eqref{pressuredefinition}, and the rate function $I\colon \mathcal{P}(\mathfrak{X})\rightarrow[0,\infty]$ by \eqref{ratedefinition}. Meanwhile, $\mathcal{V}\subset L_b(\mathfrak{X})$ and $\mathcal{W}\subset \mathcal{P}(\mathfrak{X})$ are as in Definition~\ref{def_mathcalV}. With these preparations at hand, we now state the $(\mathbf{AET})$-based Kifer's criterion.
       
    \begin{theorem}\label{Thm Kifer}
        Suppose that hypothesis \hyperlink{AET}{$\color{black}\mathbf{(AET)}$} is satisfied. Then $I$ is a good rate function, and the local LDP holds in the following sense: for any closed set $F\subset \mathcal{P}(\mathfrak{X})$ and open set $G\subset \mathcal{P}(\mathfrak{X})$,
        \begin{align}
            \limsup_{\theta\in \Theta}\frac{1}{r_\theta}\log \mathbb{P}(L_{\theta}\in F)&\le -\inf\{I(\sigma):\sigma\in F\};\label{Kifer-upperbound}\\
            \liminf_{\theta\in \Theta}\frac{1}{r_\theta}\log \mathbb{P}(L_{\theta}\in G)&\ge -\inf\{I(\sigma):\sigma\in G\cap \mathcal{W}\}.\label{Kifer-lowerbound}
        \end{align}
    \end{theorem}
    
    The rest of this subsection is devoted to the proof of Theorem~\ref{Thm Kifer}, which is divided into three steps. Asymptotic exponential tightness plays a significant role in the first two steps, leading to the goodness of $I$ and LDP upper bound \eqref{Kifer-upperbound}, respectively. The third step, which handles the lower bound \eqref{Kifer-lowerbound}, should be standard owing to ideas from \cite{DV-75}.

    \vspace{2cm}

    \noindent {\it Step 1: Goodness of the rate function.}
    
    Recall that the precompactness of the family of probability measures $\mathcal{K}_l$ is equivalent to the \textit{tightness} (see, e.g.~\cite{Billingsley-99}). Namely, for any $\varepsilon>0$, there is a compact set $K_{l,\varepsilon}\subset \mathfrak{X}$, such that
    \[\nu(K_{l,\varepsilon}^C)<\varepsilon\quad \text{for any }\nu\in \mathcal{K}_l.\]
    We shall keep the notation of $K_{l,\varepsilon}$ throughout this step.
    
    The starting point is a refined estimate for $\Lambda(V)$. We require the following simple fact, known as the Laplace principle: if $a_{\theta,1},\dots ,a_{\theta,n}>0$, then
    \begin{equation}
        \limsup_{\theta\in \Theta} \frac{1}{r_\theta}\log (a_{\theta,1}+\cdots +a_{\theta,n})=\max_{1\le k\le n} \limsup_{\theta\in \Theta} \frac{1}{r_\theta} \log a_{\theta,k}.\label{laplaceprinciple_limsup}
    \end{equation}

    \begin{lemma}
       Suppose that hypothesis \hyperlink{AET}{$\color{black}\mathbf{(AET)}$} is satisfied. Then for any $V\in C_b(\mathfrak{X})$ and $l,\varepsilon>0$,
        \begin{equation}
            \Lambda(V)\le \max \left\{(1-\varepsilon)\sup_{K_{l,\varepsilon}} V+\varepsilon \sup_\mathfrak{X} V, \sup_\mathfrak{X} V-l\right\}.\label{Lambda(V)<<V,K>}
        \end{equation}
    \end{lemma}

    \begin{proof}
        If $\nu\in \mathcal{K}_l$, then $\nu(K_{l,\varepsilon})>1-\varepsilon$ and thus
        \[\langle V,\nu\rangle\le (1-\varepsilon)\sup_{K_{l,\varepsilon}} V+\varepsilon \sup_\mathfrak{X} V.\]
        By continuity of $\nu\mapsto\langle V,\nu\rangle$ and compactness of $\mathcal{K}_l$, for any $\eta
        >0$, there exists $r>0$ such that
        \[\langle V,\nu\rangle\le (1-\varepsilon)\sup_{K_{l,\varepsilon}} V+\varepsilon \sup_\mathfrak{X} V+\eta\quad \text{for any }\nu\in B_{\mathcal{P}(\mathfrak{X})}(\mathcal{K}_l,r).\]
        Therefore
        \[\limsup_{\theta\in \Theta} \frac{1}{r_\theta} \log \mathbb{E} \mathbf{1}_{\{L_\theta\in B_{\mathcal{P}(\mathfrak{X})}(\mathcal{K}_l,r)\}}e^{r_\theta \langle V,L_\theta\rangle}\le (1-\varepsilon)\sup_{K_{l,\varepsilon}} V+\varepsilon \sup_\mathfrak{X} V+\eta.\]
        Thanks to \eqref{AET} in hypothesis \hyperlink{AET}{$\color{black}\mathbf{(AET)}$},
        \begin{align*}
            &\limsup_{\theta\in \Theta} \frac{1}{r_\theta} \log \mathbb{E} \mathbf{1}_{\{L_\theta\in B_{\mathcal{P}(\mathfrak{X})}(\mathcal{K}_l,r)^C\}}e^{r_\theta \langle V,L_\theta\rangle}\\
            \le &\sup_\mathfrak{X} V+\limsup_{\theta\in \Theta} \frac{1}{r_\theta}\log \mathbb{P}(L_\theta\in B_{\mathcal{P}(\mathfrak{X})}(\mathcal{K}_l,r)^C)\le \sup_\mathfrak{X} V-l.
        \end{align*}
        Finally, the Laplace principle \eqref{laplaceprinciple_limsup} implies
        \begin{align*}
            \Lambda(V)&=\limsup_{\theta\in \Theta} \frac{1}{r_\theta} \log \left(\mathbb{E} \mathbf{1}_{\{L_\theta\in B_{\mathcal{P}(\mathfrak{X})}(\mathcal{K}_l,r)\}}e^{r_\theta \langle V,L_\theta\rangle}+\mathbb{E} \mathbf{1}_{\{L_\theta\in B_{\mathcal{P}(\mathfrak{X})}(\mathcal{K}_l,r)^C\}}e^{r_\theta \langle V,L_\theta\rangle}\right)\\
            &\le\max \left\{(1-\varepsilon)\sup_{K_{l,\varepsilon}} V+\varepsilon \sup_\mathfrak{X} V+\eta,\sup_\mathfrak{X} V-l\right\}.
        \end{align*}
        Thus the lemma follows by sending $\eta$ to $0$.
    \end{proof}
    
    \begin{corollary}\label{Igesigma(Kepsilonl)C}
        Suppose that hypothesis \hyperlink{AET}{$\color{black}\mathbf{(AET)}$} is satisfied. Then for any $\sigma\in \mathcal{P}(\mathfrak{X})$, $l,\varepsilon>0$, 
        \begin{equation}
            I(\sigma)\ge (\sigma(K_{l,\varepsilon}^C)-\varepsilon)l.\label{Ilowerestimate}
        \end{equation}
    \end{corollary}

    \begin{proof}
        For any $r>0$, substituting $V(x):=l/r\cdot \min\{\dist_\mathfrak{X}(x,K_{l,\varepsilon}),r\}$ into \eqref{Lambda(V)<<V,K>}, we get
        \[\Lambda(V)\le \max \{(1-\varepsilon)\cdot 0+\varepsilon \cdot l,l-l\}=\varepsilon l\]
        The definition of $I$, and the fact that $V(x)=l$ whenever $\dist_\mathfrak{X}(x,K_{l,\varepsilon})\ge r$, together yield
        \[I(\sigma)\ge \langle V,\sigma\rangle-\Lambda(V)\ge \sigma(B_\mathfrak{X}(K_{l,\varepsilon},r)^C) l-\varepsilon l.\]
        Since $B_\mathfrak{X}(K_{l,\varepsilon},r)^C$ increases to $K_{l,\varepsilon}^C$ as $r\to 0$, we arrive at the required inequality.
    \end{proof}

    Let us turn to the goodness of $I$, i.e.~$\{I\le a\}$ is compact for any $a\in [0,\infty)$. Recall that $I$ is lower semicontinuous, thus $\{I\le a\}$ is closed. In view of the equivalence between precompactness and tightness of probability measures, it remains to demonstrate tightness.

    \begin{proof}[Proof of the goodness]
        If $I(\sigma)\le a$, where $a\in [0,\infty)$, then inequality \eqref{Ilowerestimate} implies,
        \[\sigma(K_{l,\varepsilon}^C)\le \frac{a}{l}+\varepsilon\quad \text{for any }l,\varepsilon>0.\]
        In particular, $\sigma(K_{2a/\varepsilon,\varepsilon
        /2}^C)\le \varepsilon$. So $\{I\le a\}$ is tight, as desired.
    \end{proof}

    \vspace{3mm}
    
    \noindent {\it Step 2: The LDP upper bound.}
    
    We aim at the upper bound \eqref{Kifer-upperbound}. As \cite{Kifer-90} deals with such upper bound for compact subsets, a typical tool to extend this result to closed subsets is exponential tightness. Our new observation is that asymptotic exponential tightness is also sufficient for this purpose. Let us first exhibit the following lemma, whose proof is straightforward and therefore omitted.

    \begin{lemma}\label{compacttrick}
        Let $X$ be a metric space. The subsets $F$, $K$ and $G$ are closed, compact and open, respectively. If $F\cap K\subset G$, then there exists $r>0$ such that $F\cap B_X (K,r)\subset G$.
    \end{lemma}

    With this observation, we can subtly modify Kifer's original proof.

    \begin{proof}[Proof of the LDP upper bound]
        To prove \eqref{Kifer-upperbound}, it suffices to show that
        \begin{equation}
            \limsup_{\theta\in \Theta} \frac{1}{r_\theta} \log \mathbb{P}(L_\theta\in F)\le -a,\label{Kifer_ubreduce}
        \end{equation}
        for any $a\in \mathbb{R}$ with $a<\inf \{I(\sigma):\sigma\in F\}$.
        
        To this end, fix any such $a\in \mathbb{R}$. Define the open set $\Gamma_V\subset \mathcal{P}(\mathfrak{X})$ for $V\in C_b(\mathfrak{X})$ by
        \[\Gamma_V:=\{\sigma\in \mathcal{P}(\mathfrak{X}):\langle V,\sigma\rangle-\Lambda(V)>a\}.\]
        For any $l\ge 0$, by definition of $\Gamma_V$, the compact set $F\cap \mathcal{K}_l$ is covered by $\{\Gamma_V:V\in C_b(\mathfrak{X})\}$. Thus there is a finite subcovering $\Gamma_{V_1},\dots, \Gamma_{V_n}$. By Lemma~\ref{compacttrick}, there exists $r>0$ such that
        \[F\cap B_{\mathcal{P}(\mathfrak{X})}(\mathcal{K}_l,r)\subset \Gamma_{V_1}\cup \cdots \cup \Gamma_{V_n}.\]
        In other words,
        \begin{equation}
            F\subset  \Gamma_{V_1}\cup \cdots \cup \Gamma_{V_n}\cup B_{\mathcal{P}(\mathfrak{X})}(\mathcal{K}_l,r)^C.\label{Kifer_ubdecompose}
        \end{equation} 
        Due to Markov's inequality, for each $1\le k\le n$,
        \begin{align}
            \limsup_{\theta\in \Theta} \frac{1}{r_\theta} \log \mathbb{P}(L_\theta\in \Gamma_{V_k})=&\limsup_{\theta\in \Theta} \frac{1}{r_\theta} \log \mathbb{P}(\langle V_k,L_\theta\rangle>\Lambda(V_k)+a)\notag\\
            \le &\limsup_{\theta\in \Theta} \frac{1}{r_\theta} \log \left(e^{-r_\theta (\Lambda(V_k)+a)} \mathbb{E} e^{r_\theta \langle V_k,L_\theta\rangle}\right)\notag\\
            =&-(\Lambda(V_k)+a)+\Lambda(V_k)=-a.\label{Kifer_ubGammaV}
        \end{align}
        Now \eqref{pressuredefinition}, \eqref{AET}, \eqref{laplaceprinciple_limsup}, \eqref{Kifer_ubdecompose} and \eqref{Kifer_ubGammaV} together imply
        \begin{align*}
            &\limsup_{\theta\in \Theta} \frac{1}{r_\theta} \log \mathbb{P}(L_\theta\in F)\\
            \le &\limsup_{\theta\in \Theta} \frac{1}{r_\theta} \log \left(\sum_{k=1}^n \mathbb{P}(L_\theta\in \Gamma_{V_k})+\mathbb{P}(L_\theta \in B_{\mathcal{P}(\mathfrak{X})}(\mathcal{K}_l,r)^C)\right)\le \max \{-a,-l\}.
        \end{align*}
        Consequently, this implies inequality~\eqref{Kifer_ubreduce} by sending $l$ to $\infty$.
    \end{proof}

    \vspace{3mm}

    \noindent {\it Step 3: The LDP lower bound.}
    
    It remains to establish the LDP lower bound \eqref{Kifer-lowerbound}. For any open set $G\subset \mathcal{P}(\mathfrak{X})$ and equilibrium state $\sigma_V\in G\cap \mathcal{W}$ with $V\in \mathcal{V}$, it suffices to derive
    \begin{align}
        \liminf_{\theta\in \Theta} \frac{1}{r_\theta} \log \mathbb{P}(L_\theta\in G)\ge -I(\sigma_V)=\Lambda(V)-\langle V,\sigma_V\rangle.\label{eslowerbound}
    \end{align}
    Let us denote with $\mu_\theta(B):=\mathbb{P}(L_\theta\in B)$ for $B\in \mathcal{B}(\mathcal{P}(\mathfrak{X}))$. Define $\mu_\theta^V\in \mathcal{P}(\mathcal{P}(\mathfrak{X}))$ by
    \begin{equation}
        \mu_\theta^V(B)=\frac{\mathbb{E} \mathbf{1}_{\{L_\theta\in B\}}e^{r_\theta \langle V,L_\theta\rangle}}{\mathbb{E} e^{r_\theta \langle V,L_\theta\rangle}}.\label{Kife_lbdefmuthetaV}
    \end{equation}

    \vspace{3mm}
    The idea of the following lemma is similar to the proof of the upper bound \eqref{Kifer-upperbound}, except that we use the compact set $\{I\le a\}$ instead of $\mathcal{K}_l$ and then send $a$ to $\infty$. It is essentially contained in \cite[Theorem 4.3.1 and Theorem 4.4.2]{DZ10}, and the proof is skipped.
    \begin{lemma}\label{lem_modifiedULDP}
        Suppose the LDP upper bound \eqref{Kifer-upperbound} holds with  a good rate function $I$. Then for any $V\in \mathcal{V}$ and closed set $F\subset \mathcal{P}(\mathfrak{X})$,
        \begin{equation}
            \limsup_{\theta\in \Theta} \frac{1}{r_\theta} \log \mu_\theta^V(F)\le -\inf\{I^V(\sigma):\sigma\in F\}.\label{modifiedupperbound}
        \end{equation}
        Here $I^V(\sigma):=I(\sigma)-\langle V,\sigma\rangle+\Lambda(V)$ is also a good rate function.
    \end{lemma}
    
    Making use of Lemma~{\rm\ref{lem_modifiedULDP}}, we establish the following technical result.
    \begin{lemma}\label{muthetaV->deltasigmaV}
        Under the assumptions of Lemma~{\rm\ref{lem_modifiedULDP}}, $\mu_\theta^V \Rightarrow \delta_{\sigma_V}$ for any $V\in \mathcal{V}$, where $\Rightarrow$ stands for the weak convergence of probability measures.
    \end{lemma}

    \begin{proof}
        It is equivalent to say, if $U\subset \mathcal{P}(\mathfrak{X})$ is a neighbourhood of $\sigma_V$, then \begin{equation}
            \lim_{\theta\in \Theta} \mu_\theta^V(U)=1.\label{Kife_lbconvergetodelta}
        \end{equation}
        
        The intersection of the compact sets $\{I^V \le 1/n\}(n\in \mathbb{N})$ is equal to the singleton $\{\sigma_V\}$. By a standard compactness argument, there exists an $n\in \mathbb{N}$ such that $\{I^V \le 1/n\}\subset U$. Thus $U^C$ is a closed set and $I^V(\sigma)>1/n$ for any $\sigma\in U^C$. In view of \eqref{modifiedupperbound}, we obtain
        \[\limsup_{\theta\in \Theta} \frac{1}{r_\theta} \log \mu_\theta^V(U^C)\le -\inf \{I^V(\sigma):\sigma\in U^C\}\le -\frac{1}{n}.\]
        Therefore $\mu_\theta^V(U^C)\to 0$, and thus $\mu_\theta^V(U)\to 1$ as desired.
    \end{proof}

    \vspace{3mm}
    Finally, we are now in a position to verify  \eqref{eslowerbound}, yielding the LDP lower bound. 
    
    \begin{proof}[Proof of the LDP lower bound]
        Since $G$ is open and $\sigma_V\in G$, for any $\varepsilon>0$, there exists $r>0$ such that $B_{\mathcal{P}(\mathfrak{X})}(\sigma_V,r)\subset G$ and 
        \begin{equation}
            \langle V,\sigma\rangle<\langle V,\sigma_V\rangle+\varepsilon\quad\text{for any }\sigma\in B_{\mathcal{P}(\mathfrak{X})}(\sigma_V,r).\label{Kifer_lbreducetonbhd}
        \end{equation}
        Gathering \eqref{Kife_lbdefmuthetaV}, \eqref{Kife_lbconvergetodelta}, \eqref{Kifer_lbreducetonbhd} and the definition of $\mathcal{V}$, we obtain
        \begin{align*}
            \liminf_{\theta\in \Theta} \frac{1}{r_\theta} \log \mathbb{P}(L_\theta\in G)&\ge \liminf_{\theta\in \Theta} \frac{1}{r_\theta} \log \mu_\theta(B_{\mathcal{P}(\mathfrak{X})}(\sigma_V,r))\\
            &=\liminf_{\theta\in \Theta} \frac{1}{r_\theta} \log \int_{B_{\mathcal{P}(\mathfrak{X})}(\sigma_V,r)} \frac{\mathbb{E} e^{r_\theta \langle V, L_\theta\rangle}}{e^{r_\theta \langle V,\sigma\rangle}}\, \mu_\theta^V(d\sigma)\\
            &=\Lambda(V)+\liminf_{\theta\in \Theta} \frac{1}{r_\theta} \log \int_{B_{\mathcal{P}(\mathfrak{X})}(\sigma_V,r)} e^{-r_\theta\langle V,\sigma\rangle}\, \mu_\theta^V(d\sigma)\\
            &\ge \Lambda(V)+\liminf_{\theta\in \Theta} \frac{1}{r_\theta} \log \left(e^{-r_\theta(\langle V,\sigma_V\rangle+\varepsilon)}\mu_\theta^V (B_{\mathcal{P}(\mathfrak{X})}(\sigma_V,r))\right)\\
            &=\Lambda(V)-\langle V,\sigma_V\rangle-\varepsilon=-I(\sigma_V)-\varepsilon.
        \end{align*}
        This immediately implies \eqref{eslowerbound}.
    \end{proof}

    Summarizing Steps 1-3, the proof of Theorem~\ref{Thm Kifer} is now complete.

    \begin{remark}\label{Rmk_Kifer}
    We also mention the original assumptions in Kifer's criterion {\rm\cite[Theorem 2.1]{Kifer-90}}, where the phase space $\mathfrak{X}$ is assumed to be compact. The precise conditions are:

    \begin{enumerate}
        \item[(1)] For any $V\in C(\mathfrak{X})$, the pressure function $\Lambda(V)$ converges.

        \item[(2)] For a dense linear subspace of $C(\mathfrak{X})$, the equilibrium state is unique. 
    \end{enumerate}
    Actually, since $V\mapsto \frac{1}{n}\log Q_n^V \mathbf{1}$ is $1$-Lipschitz, one only requires {\rm(1)} to hold for a dense subset of $C(\mathfrak{X})$. As a consequence, the two conditions are equivalent to $\mathcal{V}$ containing a dense linear subspace of $C(\mathfrak{X})$. An ideal situation for this is $\mathcal{V}=L(\mathfrak{X})$.
\end{remark}

    \subsection{Proof of Theorem~\ref{Thm LDP}}\label{subsec_prooflevel2LDP}

    We come back to the RDS setting as in the assumptions of Theorem~\ref{Thm LDP}. The first three parts are devoted to statements \hyperlink{a}{(a)}-\hyperlink{c}{(c)} on bounded invariant sets $X_R$, defined in \eqref{XR}. Then the last part quickly derives the local LDP on the whole space $\mathcal{X}$.

    \vspace{3mm}
    \noindent\textbf{Convention.} In the first three parts we denote $X_R(R>0)$ by $X$ for simplicity, and drop the label $R$ in various notations, as this parameter plays no role in the proof. Recall that  hypotheses \hyperlink{AC}{$\color{black}(\mathbf{AC})$}, \hyperlink{I}{$\color{black}(\mathbf{I})$} and \hyperlink{C}{$\color{black}(\mathbf{C})$} are still valid, with the phase space $\mathcal{X}$ replaced by $X$, the compactum $Y$ unchanged, and $V(x)$ in \eqref{asymptotic compactness} a constant depending on $R$.

    \vspace{3mm}
    
    \subsubsection{Proof of statement (b): uniform local LDP}\label{subsubsec_(b)}

    To apply Kifer's criterion, Theorem~\ref{Thm Kifer} with $\mathfrak{X}=X$, we need to verify hypothesis \hyperlink{AET}{$\color{black}(\mathbf{AET})$}, which turns out a direct consequence of \hyperlink{AC}{$\color{black}(\mathbf{AC})$}.
    
    Notice that for any $x\in X$, we have
    \[\dist_{\mathcal{P}(X)}(\delta_x,\mathcal{P}(Y))\le \dist_X (x,Y).\]
    The reason is that one can find $y\in Y$ with $d(x,y)=\dist_X (x,Y)$, and hence
    \[\dist_{\mathcal{P}(X)}(\delta_x,\mathcal{P}(Y))\le \|\delta_x-\delta_y\|_L^*=\sup_{\|f\|_L\le 1} |f(x)-f(y)|\le d(x,y).\]
    According to \eqref{asymptotic compactness} and the definition of dual-Lipschitz distance, 
    \[\dist_{\mathcal{P}(X)}(L_{n,x},\mathcal{P}(Y))\le \frac{1}{n}\sum_{k=1}^n \dist_{\mathcal{P}(X)}(\delta_{x_k},\mathcal{P}(Y))\le \frac{1}{n}\sum_{k=1}^n \dist_X(x_k,Y)\le \frac{1}{n}\sum_{k=1}^n Ce^{-\kappa k},\] 
    which converges to $0$ uniformly as $n\to \infty$. Therefore, inequality \eqref{AET} follows by letting
    \begin{equation*}
        \mathcal{K}_l=\mathcal{P}(Y)\quad\text{for any } l>0.
    \end{equation*}
    Moreover, by Corollary~\ref{Igesigma(Kepsilonl)C} and setting $K_{l,\varepsilon}=Y$ in \eqref{Ilowerestimate}, it follows that
    \begin{equation}
        I(\sigma)=\infty \quad\text{for any }\sigma\not\in\mathcal{P}(Y).\label{D(I)subsetP(Y)}
    \end{equation}

    Thus Theorem~\ref{Thm Kifer} is now applicable in our RDS~\eqref{RDS} setting. And the desired uniform level-2 local LDP on $X$, i.e.~statement \hyperlink{b}{(b)} of Theorem~\ref{Thm LDP}, follows immediately.

    \subsubsection{Proof of statement (a): rate function}\label{subsubsec_(a)}

    To establish \hyperlink{a}{(a)}, we first introduce the Feynman--Kac semigroup, which will be used to derive the variation formula for the rate function \eqref{I-variantion formula}.
    
    \begin{definition}
        For any $V\in C_b(X)$ and $n\in \mathbb{N}$, define  $Q^{V}_n\colon C_b(X)\to C_b(X)$ by
    \begin{equation}
       Q^{V}_nf(x):=\mathbb{E}_x f(x_n)e^{V(x_1)+\cdots +V(x_n)}\quad\text{for }f\in C_b(X).\label{def_F-Csemigroup}
    \end{equation}
    It is easy to see $Q^{V}_n$ indeed forms a semigroup. Denote by $Q_n^{V*}\colon \mathcal{M}_+(X)\rightarrow\mathcal{M}_+(X)$ its dual semigroup. We write $Q^V$ and $Q^{V*}$ in place of $Q_1^V$ and $Q^{V*}_1$ for simplicity.
    \end{definition}
    
    \vspace{3mm}

    If we take $f=\mathbf{1}$ in \eqref{def_F-Csemigroup}, then $Q_n^V \mathbf{1}=\mathbb{E} e^{n\langle V,L_{n,x}\rangle}$ and thus
    \begin{equation}
        \Lambda(V)=\limsup_{n\to \infty} \frac{1}{n} \log \|Q_n^V\mathbf{1}\|_\infty.\label{Lambda(V)andQ_nV}
    \end{equation}
    Therefore the Feynman--Kac semigroup has a natural connection to the large deviations for Markov processes. Since $Q_n^V$ is a semigroup and $|Q_n^V f|\le \|f\|_\infty |Q_n^V \mathbf{1}|$, we have
    \[\|Q_{m+n}^V\mathbf{1}\|_{\infty}\le \|Q_m^V\mathbf{1}\|_{\infty}\|Q_n^V\mathbf{1}\|_{\infty}.\]
    A well-known consequence of the subadditivity implies that \eqref{Lambda(V)andQ_nV} can be replaced by
    \begin{equation}
        \Lambda(V)=\lim_{n\to \infty} \frac{1}{n}\log \|Q_n^V \mathbf{1}\|_{\infty}\label{Lambda(V)andQ_nVconverge}
    \end{equation}

    In order to prove $\mathcal{I}|_{\mathcal{P}(X)}=I$, let us denote the right-hand side of \eqref{I-variantion formula} as
    \[I'\colon \mathcal{P}(X)\to [0,\infty],\quad I'(\sigma):=\begin{cases}
        -\inf\limits_{f\in C_+(Y)}\displaystyle\int_Y \log\left(\frac{Pf}{f}\right)\, d\sigma\quad&\text{for }\sigma\in\mathcal{P}(Y),\\
        \infty\quad&\text{otherwise}.
    \end{cases}\]
    Then we need to show $I=I'$. Such entropy-type formula can be found in \cite{DV-75}, where the setting is slightly different from ours. Since $P(\lambda f)=\lambda P(f)$ for any $\lambda>0$ and $Y$ is compact, the constraint $f\in C_+(Y)$ can be replaced by $f\in C(Y)$ and $f\ge 1$. 
    
    Let us introduce an intermediate term $I''\colon \mathcal{P}(X)\to [0,\infty]$, defined by
    \begin{equation}
        I''(\sigma):=-\inf\limits_{f\in C_b(X),\ f\ge 1}\int_{X}\log\left(\frac{Pf}{f}\right)\, d\sigma.\label{I2definition}
    \end{equation}
    Then following the same idea as in \cite{DV-75}, one can prove $I=I''$. For the reader's convenience, we briefly sketch the proof, borrowed from a recent resource \cite[Proposition 2.5]{JNPS-21}. 

    \begin{proof}[Sketched proof of $I=I''$]
    By virtue of \eqref{Lambdaconstant}, it suffices to derive that for any $\sigma\in \mathcal{P}(X)$,
    \[I''(\sigma)=\sup_{V\in C_b(X),\ \Lambda(V)=0} \langle V,\sigma\rangle.\]
    On the one hand, fix any such $V$ and consider (noting that $\|Q_n^V \mathbf{1}\|_\infty^{1/n}\to 1$ by \eqref{Lambda(V)andQ_nVconverge})
    \[f(x):=e^{V(x)}\sum_{n=0}^\infty e^{-\varepsilon n} Q_n^V\mathbf{1} (x),\]
    where $\varepsilon>0$ is arbitrary. Plugging such $f(x)$ into \eqref{I2definition} yields that $I''(\sigma)\ge \langle V,\sigma\rangle-\varepsilon$. 
    From this we deduce that $I''(\sigma)\ge I(\sigma)$. On the other hand, for any $f\in C_b(X)$ with $f\ge 1$, let 
    \[V(x):=-\log \left(\frac{Pf(x)}{f(x)}\right).\]
    Then using the Markov property, it can be checked that
    \[\|f\|_\infty^{-1}\le Q_n^V\mathbf{1}(x)\le \|f\|_\infty.\]
    In particular, it leads to $\Lambda(V)=0$, thereby implying $I(\sigma)\ge \langle V,\sigma\rangle$, and hence $I''(\sigma)\le I(\sigma)$.
    \end{proof}

    \begin{lemma}
         Suppose that hypothesis \hyperlink{AC}{$\color{black}\mathbf{(AC)}$} is satisfied, then the variation formula \eqref{I-variantion formula} holds, i.e.~$I(\sigma)=I'(\sigma)$ for any $\sigma\in \mathcal{P}(X)$.
    \end{lemma}

    \begin{proof}
        If $\sigma\not \in \mathcal{P}(Y)$, then \eqref{D(I)subsetP(Y)} implies that $I(\sigma)=\infty=I'(\sigma)$. And if $\sigma\in \mathcal{P}(Y)$, then it suffices to show $I'(\sigma)=I''(\sigma)$. By definition, automatically $I''(\sigma)\le I'(\sigma)$. Conversely, by Tietze's extension theorem, if $f\in C(\mathcal{A})$ and $f\ge 1$, then there is an $F\in C_b(X)$ satisfying $F|_{\mathcal{A}}=f$ and $1\le F\le \|f\|_\infty$. Thus $I'(\sigma)\le I''(\sigma)$ holds as well.
    \end{proof}

    \begin{remark}\label{I(mu*)=0}
        A by-product of the variation formula \eqref{I-variantion formula} is that $I(\mu_*)=0$. This is natural in view of the exponential mixing, but a rigorous proof is not obvious at first sight. In fact, for any $f\in C_+(Y)$, by convexity of logarithm, we have
        \[P\log f\le \log Pf.\]
        Due to the invariance of $\mu_*$,
        \[\int_Y \log f\, d\mu_*=\int_Y P\log f\, d\mu_*\le \int_Y \log Pf\, d\mu_*.\]
        Thus $I(\mu_*)\le 0$ by \eqref{I-variantion formula}. Moreover, $\mu_*$ is the unique null point of $I$. In fact, $I^{-1}(0)$ contains exactly the equilibrium states for $V=0$, which should be unique as $0\in \mathcal{V}_{L,\delta}\subset \mathcal{V}$.
    \end{remark}

    \subsubsection{Proof of statement (c): non-triviality}\label{subsubsec_(c)}
    
    We first invoke hypothesis \hyperlink{AC}{$\color{black}\mathbf{(AC)}$} to reduce the problem to $Y$, and then show $\mathcal{V}_{L,\delta}\subset \mathcal{V}$ by exploiting asymptotics of Feynman--Kac semigroups on $C(Y)$. We also show that $\mathcal{W}_R$ is independent of $R$ in Corollary~\ref{VrestricY}.

    \vspace{3mm}
    
    \noindent\textit{Step 1: Reduction to compactum.} 

    Let us show that whether $V\in \mathcal{V}$ or not depends only on its restriction to $Y$.
    \begin{lemma}\label{YimplyX}
       Suppose that hypothesis \hyperlink{AC}{$\color{black}\mathbf{(AC)}$} is satisfied. Let $V\in L_b(X)$. If the functions
        \[X\ni x\mapsto \frac{1}{n}\log (Q_n^V \mathbf{1})(x)\]
        converge uniformly to $\log \lambda$ in $Y$ as $n\to \infty$ with a constant $\lambda>0$. Then this sequence also converges uniformly to $\log \lambda$ in $X$.
    \end{lemma}

    \begin{proof}
        The estimates for upper and lower bounds are similar, hence we skip the upper bound and focus on the proof of the lower bound:
        \begin{equation}
            \liminf_{n\to \infty} \inf_{x\in X} \frac{1}{n}\log Q_n^V \mathbf{1}(x)\ge \log \lambda.\label{Lambda(V)lowerestimate}
        \end{equation}
        
        For any $0<\varepsilon<\lambda$, by the uniform convergence in $Y$, there is an $N=N(\varepsilon)\in \mathbb{N}$ such that
        \[Q_N^V \mathbf{1}(y)=\mathbb{E}_y e^{V(y_1)+\cdots +V(y_N)}\ge (\lambda-\varepsilon)^N\quad \text{for any }y\in Y.\]
        If $x\in X$, then there exists $y\in Y$ satisfying $\dist_X(x,Y)=d(x,y)$. By the  Lipschitz continuity of $S(\cdot,\zeta)$, for $k=1,\dots, N$, almost surely,
        \[d(x_k,y_k)=d(S(x_{k-1},\xi_{k-1}), S(y_{k-1},\xi_{k-1}))\le Cd(x_{k-1}, y_{k-1})\le C^k d(x,y).\]
        We keep the convention that various constants $C$ in the proof, which may differ from line to line, are independent of $x$ and $n$. Additionally, they may depend on other parameters such as $\varepsilon$ and $N$. Together with the Lipschitz continuity of $V$, we obtain that almost surely,
        \begin{align}
            &e^{V(x_1)+\cdots +V(x_N)}=e^{V(y_1)+\cdots +V(y_N)}+\left(e^{\sum_{k=1}^n (V(x_k)-V(y_k))}-1\right)e^{V(y_1)+\cdots +V(y_n)}\notag\\\ge &(\lambda-\varepsilon)^N-C\left|e^{\sum_{k=1}^N (V(x_k)-V(y_k))}-1\right|\ge (1-C\dist_X(x,Y)) (\lambda-\varepsilon)^N.\label{Nstepestimate}
        \end{align}
        
        For any $x\in X$ and $n\in \mathbb{N}$, suppose $mN\le n<(m+1)N$, and consider the expected value with respect to $\mathcal{F}_{n-N}$, where $\mathcal{F}_k$ is the $\sigma$-algebra generated by $\xi_1,\dots, \xi_k$. Taking \eqref{asymptotic compactness} and \eqref{Nstepestimate} into account, we obtain
        \begin{align}
            Q_n^V \mathbf{1}(x)&\ge \mathbb{E}_x \left(e^{V(x_1)+\cdots +V(x_{n-N})}\mathbb{E}(e^{V(x_{n-N+1})+\cdots +V(x_n)}|\mathcal{F}_{n-N})\right)\notag\\
            &\ge (\lambda-\varepsilon)^N\mathbb{E}_x ((1-C\dist_X(x_{n-N},Y))e^{V(x_1)+\cdots +V(x_{n-N})})\notag\\
            &\ge (1-Ce^{-\kappa (n-N)})(\lambda-\varepsilon)^N Q_{n-N}^V \mathbf{1}(x).\label{nton-Nestimate}
        \end{align}
        In order to repeat this procedure, we need $Ce^{-\kappa (n-N)}<1$. There is an integer $M\ge 1$ such that $Ce^{-\kappa k}<1$ provided $k\ge (M-1)N$. Thus we can repeat \eqref{nton-Nestimate} for $m-M$ times and attain
        \begin{equation*}
            Q_n^V \mathbf{1}(x)\ge R_n(\lambda-\varepsilon)^{(m-M)N}\mathbb{E}_x e^{V(x_1)+\cdots +V(x_{n-(m-M)N})}\ge CR_n(\lambda-\varepsilon)^n,
        \end{equation*}
        where $R_n$ denotes for
        \[R_n:=\prod_{k=1}^{m-M} (1-Ce^{\kappa(n-kN)})\ge \prod_{j=(M-1)N}^\infty (1-Ce^{-\kappa j})>0. \]
        As a result, it follows that $Q_n^V\mathbf{1}\ge C(\lambda-\varepsilon)^n$, which immediately yields \eqref{Lambda(V)lowerestimate}.
    \end{proof}

    Note that $Q_n^V$ can be defined on $C(Y)$, for $Y$ is invariant, and that $I(\sigma)=\infty$ if $\sigma\not \in \mathcal{P}(Y)$.
    
    \begin{definition}\label{def_V(Y)}
        The family $\mathcal{V}(Y)$ consists of $V\in L(Y)$ with the property that, there exists a unique $\sigma_V\in \mathcal{P}(Y)$ such that
        \[\frac{1}{n}\log Q_n^V \mathbf{1}\rightarrow\langle V,\sigma_V\rangle-I(\sigma_V)\quad \text{in }C(Y).\]
        In other words, this is equivalent to the following two conditions:
        \begin{enumerate}
            \item[(1)] The sequence of functions $\frac{1}{n}\log Q_n^V \mathbf{1}$ converges uniformly in $Y$ to a constant $\Lambda(V)$.

            \item[(2)] There exists a unique $\sigma_V\in \mathcal{P}(Y)$ (still called the equilibrium state), such that
            \[\Lambda(V)=\langle V,\sigma_V\rangle-I(\sigma_V).\]
        \end{enumerate}
    \end{definition}
    
    The following corollary is straightforward. For clarity we restore the label $R$ for the moment.
    \begin{corollary}\label{VrestricY}
         Suppose that hypothesis \hyperlink{AC}{$\color{black}\mathbf{(AC)}$} is satisfied. If $V\in L_b(X_R)$, then $V\in \mathcal{V}_R$ if and only if $V|_Y\in \mathcal{V}(Y)$. Moreover, $\mathcal{W}_R$ is independent of $R$, and can be characterized by
        \[\mathcal{W}_R=\{\sigma_V:V\in \mathcal{V}(Y)\}.\]
    \end{corollary}

    \begin{proof}
        The first assertion follows from Lemma~\ref{YimplyX}, and that $I(\sigma)=\infty$ for any $\sigma\not \in \mathcal{P}(Y)$. The second assertion is due to the fact that the equilibrium state for $V$ is the same as the one for $V|_Y$, since the equilibrium state $\sigma_V\in \mathcal{P}(Y)$.
    \end{proof}

    Let us define $\mathcal{V}_{L,\delta}(Y)$ by
    \[\mathcal{V}_{L,\delta}(Y):=\{V\in L(Y):\Lip(V)<L,\ \Osc(V)<\delta\}.\]
    To finish the proof of Theorem~\ref{Thm LDP}, it suffices to show $\mathcal{V}_{L,\delta}(Y)\subset \mathcal{V}(Y)$ for $\delta$ sufficiently small.

    \vspace{3mm}
    
    \noindent\textit{Step 2: Long-time behaviour of Feynman--Kac semigroups}
    
    The proof of $\mathcal{V}_{L,\delta}(Y)\subset \mathcal{V}(Y)$ can be further reduced to the study of asymptotics of the Feynman--Kac semigroups. From now on, $Q_n^V$ is viewed as an operator on $C(Y)$. The following proposition will be demonstrated in Appendix~\ref{sec_verify}.
    
    \begin{proposition}\label{Prop Feynman-Kac stability} Under the assumptions of Theorem {\rm\ref{Thm LDP}}, given $L>0$ one can find $\delta>0$ with the following properties: for any $V\in \mathcal{V}_{L,\delta}(Y)$, there exists $\lambda_V>0$, $h_V\in C_+(Y)$ and $\mu_V\in\mathcal{P}(Y)$ with $\supp(\mu_V)=\mathcal{A}$, such that for any $f\in C(Y)$ and $\nu\in\mathcal{M}_+(Y)$,
    \begin{align}
        &Q^{V}h_V=\lambda_Vh_V,\quad  Q^{V*}\mu_V=\lambda_V\mu_V,\quad\langle h_V,\mu_V\rangle=1,\label{Feynman--Kac invariance}\\
        &\lambda_V^{-n}Q_n^{V}f\rightarrow\langle f,\mu_V\rangle h_V\quad\text{ in }C(Y)\text{ as }n\rightarrow\infty,\label{Feynman--Kac f-convergence}\\
        &\lambda_V^{-n}Q_n^{V*}\nu\rightarrow\langle h_V,\nu\rangle \mu_V\quad\text{ in }\mathcal{M}_+(Y)\text{ as }n\rightarrow\infty.\label{Feynman--Kac m-convergence}
    \end{align}
    \end{proposition}

    \begin{remark}
        The eigenvalue $\lambda_V$ and eigenvectors $\mu_V$ and $h_V$ satisfying \eqref{Feynman--Kac invariance}-\eqref{Feynman--Kac m-convergence} are unique. In fact, since $h_V\in C_+(Y)$, taking $f=\mathbf{1}$ in \eqref{Feynman--Kac f-convergence} leads to
        \begin{equation}
            \frac{1}{n}\log Q_n^V \mathbf{1}\to \log \lambda_V\quad \text{in }C(Y)\text{ as }n\to \infty.\label{convergenceonY}
        \end{equation}
        So $\lambda_V=e^{\Lambda(V)}$ is uniquely determined. Moreover, again by \eqref{Feynman--Kac f-convergence}, $\lambda_V^{-n}Q_n^V \mathbf{1}\to h_V$ in $C(Y)$. Similarly the uniqueness of $\mu_V$ follows from \eqref{Feynman--Kac m-convergence}.
    \end{remark}

    Let us assume for the moment the validity of Proposition~\ref{Prop Feynman-Kac stability}, and justify $\mathcal{V}_{L,\delta}(Y)\subset \mathcal{V}(Y)$.

    \begin{lemma}\label{nontrivialonY}
        Under the assumptions of Theorem {\rm\ref{Thm LDP}}, given any $L>0$, let $\delta>0$ be specified as in Proposition~{\rm\ref{Prop Feynman-Kac stability}}, then $\mathcal{V}_{L,\delta}(Y)\subset \mathcal{V}(Y)$.
    \end{lemma}

    \begin{proof}
        Fix any $V\in \mathcal{V}_{L,\delta}(Y)$. Recall the two conditions in Definition~\ref{def_V(Y)}. Clearly the first condition is a consequence of \eqref{convergenceonY}. Next we check the uniqueness of the equilibrium state. The idea originates from \cite{DV-75}. By virtue of a variation method, it turns out that 
    \[\sigma_V=h_V\mu_V,\]
    where $h_V$ and $\mu_V$ are the eigenvectors for $Q^V$ and $Q^{V*}$ in Proposition~\ref{Prop Feynman-Kac stability}, respectively. Since this argument is already known, we only sketch the proof.
    
    Define the semigroup $T_n$ on $C(Y)$ by
    \[T_nf:=\lambda_V^{-n} h_V^{-1} Q_n^V (h_Vf).\]
    It is easy to see $T_n$ is a Markov semigroup, and
    \begin{equation*}
        T_n^*\sigma=\lambda_V^{-n} h_V Q_n^{V*} (h_V^{-1}\sigma)
    \end{equation*}
    A combination of the following two assertions readily leads to $\sigma_V=h_V\mu_V$: \begin{enumerate}
        \item[(1)] The unique invariant measure of $T^*$ is $h_V\mu_V$.

        \item[(2)] The equilibrium state $\sigma_V$ is an invariant measure of $T^*$.
    \end{enumerate}
    In fact, (1) is simple since $T$ is a conjugation of $Q^V$. As for (2), note that
    \[\Lambda(V)=\langle V,\sigma_V\rangle-I(\sigma_V)=\langle V,\sigma_V\rangle+\inf_{f\in C_+(Y)} \int_Y \log \left(\frac{Pf}{f}\right)\, d\sigma_V.\]
    One can check by manipulation that the infimum on the right-hand side is attained at $f=e^V h_V$. Taking variation with respect to any $g\in C(Y)$, we get
    \[\int_Y \frac{Pg}{Pf}\, d\sigma_V=\int_Y \frac{g}{f}\, d\sigma_V.\]
    Replacing $g$ by $fg$, thanks to $Pf=\lambda_V h_V$ and $P(fg)=Q^V (h_V g)$, we conclude
    \[\int_Y Tg\, d\sigma_V=\int_Y g\, d\sigma_V.\]
    Thus $T^* \sigma_V=\sigma_V$ and (2) follows.
    \end{proof}

    So far we have shown that the proof of $\mathcal{V}_{L,\delta}(Y)\subset \mathcal{V}(Y)$ can be reduced to Proposition~\ref{Prop Feynman-Kac stability}. Indeed, $Q_n^V$ is a special case of the generalized Markov semigroups introduced in Appendix~\ref{sec_multiplicativeergodic}. Then we manage to exploit multiplicative ergodic theorems (Theorem~\ref{Thm Feynman-Kac} and Theorem~\ref{Thm Feynman-Kac_Y}), established in \cite{JNPS-15,MN-20}. We follow the arguments from \cite{KS-00,JNPS-15,MN-20} and present the proof in the appendix; see Appendix~\ref{sec_verify}.
    
    \subsubsection{Local LDP on the whole space}

    To finish the demonstration of Theorem~\ref{Thm LDP}, we combine some simple observations to extend the LDP from bounded sets $X_R$ to the whole space $\mathcal{X}$.

    \begin{proof}[Proof of local LDP on $\mathcal{X}$]
        For any closed set $F\subset \mathcal{P}(\mathcal{X})$, since $F\cap \mathcal{P}(X_R)$ is a closed subset in $\mathcal{P}(X_R)$, and the rate function $\mathcal{I}=I_R$ on $\mathcal{P}(X_R)$, we have
    \begin{align*}
        &\limsup_{n\to \infty}\sup_{\boldsymbol{u}\in X_R} \frac{1}{n}\log \mathbb{P}(L_{n,\boldsymbol{u}}\in F)=\limsup_{n\to \infty} \sup_{\boldsymbol{u}\in X_R} \frac{1}{n}\log \mathbb{P}(L_{x,\boldsymbol{u}}\in F\cap \mathcal{P}(X_R))\\
        \le &-\inf \{I_R(\sigma):\sigma\in F\cap \mathcal{P}(X_R)\}\le -\inf \{\mathcal{I}(\sigma):\sigma\in F\}.
    \end{align*}
    This gives us the LDP upper bound \eqref{level2XR_upperbound}, which is uniform with respect to $x\in X_R$.
    
    For the lower bound \eqref{level2XR_lowerbound}, since $\mathcal{W}\subset \mathcal{P}(Y)\subset \mathcal{P}(X_R)\subset\mathcal{P}(\mathcal{X})$, for any open set $G\subset \mathcal{P}(\mathcal{X})$, 
    \begin{align*}
        &\liminf_{n\to \infty}\inf_{x\in X_R} \frac{1}{n}\log \mathbb{P}(L_{n,x}\in G)=\liminf_{n\to \infty} \inf_{x\in X_R} \frac{1}{n}\log \mathbb{P}(L_{n,x}\in G\cap \mathcal{P}(X_R))\\
        \ge &-\inf \{I_R(\sigma):\sigma\in G\cap \mathcal{P}(X_R)\cap \mathcal{W}\}=-\inf \{\mathcal{I}(\sigma):\sigma\in G\cap \mathcal{W}\}.\qedhere
    \end{align*}
    \end{proof}
    
    Now the proof of Theorem~\ref{Thm LDP} is complete.

    \vspace{2cm}
    \section*{Appendix}\label{sec_appendix}

    \stepcounter{section}
    
    \numberwithin{equation}{subsection}
    \numberwithin{theorem}{subsection}
    
    \setcounter{equation}{0}
    \setcounter{theorem}{0}
    
    \renewcommand\thesubsection{\Alph{subsection}}

    \stepcounter{subsection}

    \noindent\Alph{subsection}. \,\textbf{Some auxiliary results.}
    \subsubsection{Probability measures in a subset}\label{subsec_McShane}

    For the sake of rigour, we include a discussion on the dual-Lipschitz norm in a subset. The following result on the extension of Lipschitz functions is sometimes called McShane's lemma. To specify the domain, we denote
    \[\Lip(f;Y):=\sup_{y,y'\in Y} \frac{|f(y)-f(y')|}{d(y,y')}.\]

    \begin{lemma}\label{McShane}
        Let $(X,d)$ be a metric space. For any closed subset $Y$ and $f\in L(Y)$, there exists a function $F\in L(X)$, such that 
        \[F|_Y=f,\quad \Lip(F;X)=\Lip(f;Y)\quad \text{and}\quad \inf_Y f\le F(x)\le \sup_Y f.\] In particular, if $f\in L_b(Y)$, then $F\in L_b(X)$ and $\|F\|_{L_b(X)}=\|f\|_{L_b(Y)}$.
    \end{lemma}

    \begin{proof}
        One can check the following $F$ meets the first two requirements:
        \[F(x):=\inf\{f(y)+\Lip(f;Y)d(x,y):y\in Y\}.\]
        Then replace $F$ by $\min\{\max\{F,\inf_Y f\},\sup_Y f\}$.
    \end{proof}

    As a consequence, the dual-Lipschitz distance is not affected by restricting to a closed subspace, as long as it includes the support of the measure. The proof is trivial and skipped.
        
    \begin{corollary}
        Let $X$ be a Polish space, and $Y\subset X$ is a closed subset. Then
        \[\|\mu-\nu\|_{L(Y)}^*=\|\mu-\nu\|_{L(X)}^*,\]
        for any $\mu,\nu\in \mathcal{P}(Y)\subset \mathcal{P}(X)$.
    \end{corollary}

    Hence the natural inclusion $\mathcal{P}(Y)\subset \mathcal{P}(X)$ is continuous with respect to the weak-* topology. In particular, if $Y$ is compact, then $\mathcal{P}(Y)$ is a compact subset of $\mathcal{P}(X)$.

    \subsubsection{Central limit theorems for Markov processes}\label{sec_clt}
    
    Let $(x_n)_{n\in \mathbb{N}_0}$ be a Markov chain on the Polish space $X$, which is assumed to be exponentially mixing in the sense of \eqref{Exponential mixing}. 
    \begin{proposition}\label{prop CLT}
        Under the above assumptions, if $\mathscr{D}(x_0)=\mu_*$, then for any $f\in L_b(X)$,
        \begin{equation}\label{CLT}
            \mathscr{D} \left(\frac{1}{\sqrt{n}}\sum_{k=1}^{n}(f(x_k)-\langle f,\mu_*\rangle)\right)\Rightarrow N(0,\sigma_f^2)\quad\text{as }n\rightarrow\infty,
        \end{equation}
        where $N(0,\sigma_f^2)$ is the normal distribution with mean zero and variance $\sigma_f^2\ge 0$.
        
        Moreover, $\sigma_f>0$ provided the following two conditions are satisfied for some $z\in X$: 
        
        \begin{enumerate}
            \item[(1)] For any  $\varepsilon>0$, there exists $N\in\mathbb{N}$, such that for each $k\in \mathbb{N}$,
            \begin{equation}\label{CLT-I}
                \mathbb{P}\left(x_n\in B_X(z,\varepsilon)\text{ for any }N\le n\le N+k\right)>0.
            \end{equation}  

            \item[(2)] $f(z)\neq \langle f,\mu_*\rangle$.
        \end{enumerate}
    \end{proposition}

    \begin{proof}
        The central limit theorem \eqref{CLT} has been derived in \cite[Proposition 2.1]{LWXZZ24}. For the rest of the proof, let us assume that $\sigma_f=0$ and $\langle f,\mu_*\rangle=0$, and aim to show that $f(z)=0$. Consider the martingale approximation given by
        \begin{equation}\label{martingale}
            \sum_{k=1}^{n}f(x_k)=M_n+N_n,\quad M_n=\phi(x_{n})-\phi(x)+\sum_{k=1}^n f(x_k),\quad N_n=\phi(x)-\phi(x_{n}),
        \end{equation}
        where $\phi(x):=\sum_{k=1}^{\infty}P_kf(x)$. Following the arguments in \cite[Proposition 4.1.4]{KS-12}, one attains
        \begin{equation*}
            \mathbb{P}(M_n=0\text{ for any } n\in \mathbb{N})=1.
        \end{equation*}
    
        We claim that this equality implies that $f(z)=0$. Otherwise, without loss of generality, assume that $f(z)>0$. Let $\varepsilon>0$ be such that $f(x)>\varepsilon$ for any $x\in B(z,\varepsilon)$. By virtue of \eqref{CLT-I} and the Markov property, for any $n>N$,
        \begin{equation*}
            \mathbb{P}\left(\sum_{k=1}^{n}f(x_k)>(n-N)\varepsilon-N\|f\|_{\infty}\right)>0.
        \end{equation*}
        On the other hand, the exponential mixing implies that $\phi\in L^\infty$. For sufficiently large $n$, we see that \eqref{martingale} cannot hold almost surely, completing the proof. 
    \end{proof}

    \subsubsection{Asymptotics of generalized Markov semigroups}\label{sec_multiplicativeergodic}

    We quote two results from \cite{JNPS-15,MN-20} on the long-time behaviour of generalized Markov semigroups.
    
    Let $Y$ be a compact metric space. A generalized Markov kernel consists of a family
    \[\{Q(y,\cdot):y\in Y\}\subset \mathcal{M}_+(Y),\]
    such that $0<Q(y, Y)<\infty$ for any $y\in Y$, and the mapping $y\mapsto Q(y,\cdot)$ is continuous from $Y$ to $\mathcal{M}_+(Y)$. We denote by $Q_n(y,\cdot)$ the $n$-fold iteration of $Q(y,\cdot)$, i.e.~
    \[Q_n(y,\cdot):=\int_Y Q_{n-1}(y',\cdot)\, Q(y,dy').\]
    With a slight abuse of notation, let the operators
    \begin{equation*}
        Q_n\colon C(Y)\rightarrow C(Y),\quad Q_n^*\colon \mathcal{M}_+(Y)\rightarrow \mathcal{M}_+(Y)
        \end{equation*}
    be defined by
    \begin{equation*}
    Q_nf(y)=\int_{Y}f(z)\, Q_n(y,dz),\quad Q_n^*\sigma(\cdot )=\int_{Y}Q_n(y,\cdot )\, \sigma(dy),
    \end{equation*}
    where $f\in C(Y)$ and $\sigma\in\mathcal{M}_+(Y)$. 
    
    Let $A\subset Y$ be an invariant compact subset, in the sense that $Q(a,Y\setminus A)=0$ for any $a\in A$. Thus $Q_n$ and $Q_n^*$ can also be viewed as operators on $C(A)$ and $\mathcal{M}_+(A)$. The first theorem deals with the asymptotics of the restriction of $Q_n$ to $A$, and is established in \cite[Theorem 2.1]{JNPS-15}.
        
    \begin{theorem}\label{Thm Feynman-Kac}
        Under the above assumptions, let us assume the following conditions:
        \begin{itemize}
            \item[\tiny$\bullet$] {\bf Uniform irreducibility.} For any $\varepsilon>0$, there exists $N\in \mathbb{N}$ and $p>0$, such that
            \begin{equation}\label{uniformirreducible_A}
                Q_N(b,B_A(a,\varepsilon))\ge p\quad\text{for any }a,b\in A.
            \end{equation}
               
            \item[\tiny$\bullet$] {\bf Uniform Feller property.} For any $f\in L(A)$, the sequence $(\|Q_n\mathbf{1}\|_{L^\infty(A)}^{-1}Q_nf)_{n\in \mathbb{N}}$ is equicontinuous in $C(A)$.
        \end{itemize}
        Then there is a number $\lambda>0$, a measure $\mu\in\mathcal{P}(A)$ with $\supp(\mu)=A$, and a positive function $h\in C_+(A)$, such that for any $f\in C(A)$ and $\nu\in\mathcal{M}_+(A)$,
        \begin{align*}
            Qh&=\lambda h,\quad Q^*\mu=\lambda\mu,\quad \langle h,\mu\rangle =1,\\
            \lambda^{-n}Q_nf&\rightarrow\langle f,\mu\rangle h\quad\text{in }C(A)\text{ as }n\rightarrow\infty,\\
            \lambda^{-n}Q_n^*\nu&\rightarrow\langle h,\nu\rangle \mu\quad \text{in }\mathcal{M}_+(A)\text{ as }n\rightarrow\infty.
        \end{align*}
    \end{theorem}

    The second theorem is an extension to $Y$ and is established in \cite[Theorem 1.2]{MN-20}\footnote{It seems to us that there is a tiny inaccuracy on \eqref{A-concentration} in the original statement of \cite{MN-20}, which is corrected here. The same proof still works, hence is also omitted.}. Now the generalized Markov semigroup is only irreducible on $A$. Some extra assumptions are needed.
    
    \begin{theorem}\label{Thm Feynman-Kac_Y}
        Under the assumptions of Theorem~{\rm\ref{Thm Feynman-Kac}}, let $A\subset Y$ be an invariant compact subset. Assume the following conditions hold:
        \begin{itemize}
            \item[\tiny$\bullet$] {\bf Uniform irreducibility to $A$.} For any $\varepsilon>0$, there exists $N\in \mathbb{N}$ and $p>0$, such that
            \begin{equation}\label{uniformirreducible_Y}
                Q_N(y,B_Y(a,\varepsilon))\ge p\quad\text{for any }y\in Y\text{ and }a\in A.
            \end{equation}
               
            \item[\tiny$\bullet$] {\bf Uniform Feller property.} For any $f\in L(Y)$, the sequence $(\|Q_n\mathbf{1}\|_{L^\infty(Y)}^{-1}Q_nf)_{n\in \mathbb{N}}$ is equicontinuous in $C(Y)$.

            \item[\tiny$\bullet$] {\bf Concentration near $A$.} For any $r>0$, 
            \begin{equation}\label{A-concentration}
                \lim_{n\to \infty} \lambda^{-n} \|Q_n(\cdot , Y\setminus B_Y(A,r))\|_{L^\infty(Y)}=0.
            \end{equation}

            \item[\tiny$\bullet$] {\bf Exponential boundedness.} We have
            \begin{equation*}
                \sup_{n\in \mathbb{N}} \lambda^{-n} \|Q_n\mathbf{1}\|_{L^\infty(Y)}<\infty. 
            \end{equation*}
        \end{itemize}
        Then $h$ admits an extension in $C_+(Y)$, which is still denoted by $h$, such that for any $f\in C(Y)$ and $\nu\in\mathcal{M}_+(Y)$,
        \begin{align*}
            Qh&=\lambda h,\quad Q^*\mu=\lambda\mu,\quad \langle h,\mu\rangle =1,\\
            \lambda^{-n}Q_nf&\rightarrow\langle f,\mu\rangle h\quad\text{in }C(Y)\text{ as }n\rightarrow\infty,\\
            \lambda^{-n}Q_n\nu&\rightarrow\langle h,\nu\rangle \mu\quad\text{in }\mathcal{M}_+(Y)\text{ as }n\rightarrow\infty.
        \end{align*}
    \end{theorem}

    \stepcounter{subsection}

    \setcounter{equation}{0}
    \setcounter{theorem}{0}
    
    \noindent\Alph{subsection}. \,\textbf{Proof of some assertions.}
    
    \subsubsection{Proof of Corollary~{\rm\ref{Cor level1LDP}}}\label{sec_prooflevel1}

    The proof is in the spirit of contraction principle; see, e.g.~\cite[Theorem 4.2.1]{DZ10}. Recall the rate function $\mathcal{I}\colon \mathcal{P}(\mathcal{X})\to [0,\infty]$ for the level-2 LDP (Theorem~\ref{Thm LDP}) is defined by \eqref{I-variantion formula}. Given $f\in L_b(\mathcal{X})$, let us define $\mathcal{I}^f\colon \mathbb{R}\to [0,\infty]$ as
    \[\mathcal{I}^f(p):=\inf \{\mathcal{I}(\sigma): \sigma\in \mathcal{P}(\mathcal{X}),\ \langle f,\sigma\rangle=p\}.\]
    If $\mathcal{I}^f(p)<\infty$, then the goodness of $\mathcal{I}$ readily implies that there exists $\sigma_p\in \mathcal{P}(\mathcal{X})$ (not necessarily unique) such that $\langle f,\sigma_p\rangle=p$ and $\mathcal{I}^f(p)=\mathcal{I}(\sigma_p)$. From this as well as the convexity of $\mathcal{I}$, it readily follows that $\mathcal{I}^f$ is a convex good rate function.
    
    Define $\Lambda^f\colon \mathbb{R}\to \mathbb{R}$ by
    \[\Lambda^f(\beta)=\Lambda_R(\beta f|_{X_R}).\]
    This is well-defined, as the right-hand side is independent of $R$ (recall $\{I<\infty\}\subset \mathcal{P}(Y)$):
    \[\Lambda_R(\beta f|_{X_R})=\sup_{\sigma\in \mathcal{P}(X_R)} (\langle \beta f,\sigma\rangle-I_R(\sigma))=\sup_{\sigma\in \mathcal{P}(Y)} (\langle \beta f,\sigma\rangle-\mathcal{I}(\sigma)).\]
    Then $\Lambda^f$ is a convex function, and hence continuous. Note that
    \begin{align*}
        \sup_{p\in \mathbb{R}} (p\beta-\mathcal{I}^f(p))&=\sup_{p\in \mathbb{R}} \{p\beta-\mathcal{I}(\sigma):\sigma\in \mathcal{P}(X),\ \langle f,\sigma\rangle=p\}\\
        &=\sup_{\sigma\in \mathcal{P}(\mathcal{X})} (\beta\langle f,\sigma\rangle-\mathcal{I}(\sigma))=\Lambda^f(\beta).\qedhere
    \end{align*}
    In other words, $\Lambda^f$ is the Legendre transform of $\mathcal{I}^f$. Thus the duality relation holds:
    \[\mathcal{I}^f(p)=\sup_{\beta\in \mathbb{R}} (p\beta-\Lambda^f(\beta)).\]
    
    With a slight abuse of terminology, we call $p\in \mathbb{R}$ an equilibrium state for $\beta\in \mathbb{R}$ if
    \[\Lambda^f (\beta)=p\beta-\mathcal{I}^f(p).\]
    In view of the convexity of $\Lambda^f$, this is equivalent to
    \[p\in [D^-\Lambda^f (\beta),D^+ \Lambda^f (\beta)].\]
    Here $D^\pm$ refers to the left/right derivative, respectively. Therefore, such equilibrium state always exists, and the uniqueness corresponds to the differentiability of $\Lambda^f$ at $\beta$.

    If $p$ is an equilibrium state for $\beta$, then
    \[\Lambda^f(\beta)=p\beta-\mathcal{I}^f(p)=\langle \beta f,\sigma_p\rangle-\mathcal{I}(\sigma_p).\]
    Thus $\sigma_p$ is an equilibrium state for $\beta f$. According to Theorem~\ref{Thm LDP}, there exists $\beta_0>0$, such that for any $\beta$ with $|\beta|\le \beta_0$, we have $\beta f|_Y\in \mathcal{V}_{L,\delta}(Y)\subset \mathcal{V}(Y)$ and hence it admits a unique equilibrium state. And in particular, such $\beta$ also admits a unique equilibrium state $p=D\Lambda^f (\beta)$. Recall that the derivative of a convex function is monotone increasing, and cannot have jump discontinuity. Therefore $D\Lambda^f$ is well-defined and continuous in $[-\beta_0,\beta_0]$. Let us define
    \[J:=[D\Lambda^f (-\beta_0),D\Lambda^f (\beta_0)].\]
    
    The argument of the following lemma can be found in \cite[Section 2]{MN-18}.
    \begin{lemma}\label{Jisnonempty}
        Under the assumptions of Theorem~{\rm\ref{Thm LDP}}, $D\Lambda^f(0)=\langle f,\mu_*\rangle$ and belongs to $J^\circ$.
    \end{lemma}

    \begin{proof}
        One readily finds that $D\Lambda^f (0)=\langle f,\mu_*\rangle$ is equivalent to $\mathcal{I}(\mu_*)=0$ (see Remark~\ref{I(mu*)=0}).
        
        It remains to derive $\langle f,\mu_*\rangle\in J^\circ$. Let us claim that if $V\in \mathcal{V}(Y)$ and $\Lambda (V)=\langle V,\mu_*\rangle$, then $V(z)=\langle V,\mu_*\rangle$. Once the claim is true, we must have $D\Lambda^f (\beta_0)>D\Lambda^f (0)=\langle f,\mu_*\rangle$. Otherwise for any $\beta\in (0,\beta_0)$, the equilibrium state is equal to $0$ and $\Lambda^f(\beta)=\langle \beta f,\mu_*\rangle$, contradicting to $f(z)\not =\langle f,\mu_*\rangle$. Similarly $D\Lambda^f (-\beta_0)<D\Lambda^f (0)$. Thus $\langle f,\mu_*\rangle$ belongs to the interior of $J$.
        
        To verify our claim, replacing $V$ by $V-\langle V,\mu_*\rangle$, we can assume $\Lambda (V)=\langle V,\mu_*\rangle=0$. Since $\lambda_V=e^{\Lambda (V)}=1$, Proposition~\ref{Prop Feynman-Kac stability} implies $\|Q_n^V \mathbf{1}\|_{L^\infty (\mathcal{A})}\le C$. As a result, if the initial distribution is already stationary, i.e.~$\mathscr{D}(x)=\mu_*$, then it follows that
        \[\mathbb{E}_{\mu_*} e^{V(x_1)+\cdots +V(x_n)}\le C\quad \text{for any }n\in \mathbb{N}.\]
        Taking the central limit theorem (Proposition~\ref{prop CLT}) into account, we must have $V(z)=0$.
    \end{proof}

    We apply the contraction principle to attain the following lemma. The only difference between this and the standard contraction principle is that the lower bound is local. The proof can be easily modified, since for any $p\in G\cap J$, there exists $\sigma_p\in \mathcal{W}$ such that $\langle f, \sigma_p\rangle=p$ and $\mathcal{I}(\sigma_p)=\mathcal{I}^f(p)$. This observation also implies $\mathcal{I}^f (p)<\infty$ for $p\in J$. The proof is thus omitted.
    
    \begin{lemma}
        Under the assumptions of Theorem~{\rm\ref{Thm LDP}}, for any $f\in L_b(\mathcal{X})$, the rate function $\mathcal{I}^f$ is good. For any closed set $F\subset \mathbb{R}$ and any open set $G\subset \mathbb{R}$,
        \begin{align*}
            \limsup_{n\to \infty} \frac{1}{n}\log \mathbb{P}_x \left(\frac{1}{n}\left(f(x_1)+\cdots +f(x_n)\right)\in F\right)&\le -\inf \{\mathcal{I}^f(p):p\in F\},\\
            \liminf_{n\to \infty} \frac{1}{n}\log \mathbb{P}_x\left(\frac{1}{n}\left(f(x_1)+\cdots +f(x_n)\right)\in G\right)&\ge -\inf \{\mathcal{I}^f(p):p\in G\cap J\}.
        \end{align*}
        Moreover, the convergences are locally uniform with respect to $x\in \mathcal{X}$.
    \end{lemma}

    Now we complete the proof of the level-1 LDP.
    
    \begin{proof}[Proof of Corollary~{\rm\ref{Cor level1LDP}}]
        This follows from the lemma above and a simple observation:
        \[\inf \{\mathcal{I}^f (p):p\in B\}=\inf \{\mathcal{I}^f(p):p\in \overline{B}\}=\inf \{\mathcal{I}^f (p):p\in B^\circ\},\]
        owing to $I^f$ being finite, convex and continuous on $J$, as well as $B\subset \overline{B^\circ}$.
    \end{proof}

    \subsubsection{Proof of Proposition~{\rm\ref{Prop Feynman-Kac stability}}}\label{sec_verify}

    We only need to verify various assumptions in Theorem~\ref{Thm Feynman-Kac} and Theorem~\ref{Thm Feynman-Kac_Y}. The arguments can be found in \cite{KS-00,JNPS-15,MN-20}.
    
    \vspace{3mm}
    \noindent{\it Step 1. Uniform irreducibility.} Since \eqref{uniformirreducible_A} is contained in \eqref{uniformirreducible_Y}, and due to $Q_n^V\ge e^{n\inf_X V} P_n$, it suffices to prove that for any $\varepsilon>0$, there exist $N\in \mathbb{N}$ and $p>0$ such that
    \begin{equation}\label{UIinequality}
        P_N(y,B_Y(a,\varepsilon))\ge p\quad\text{for any }y\in Y\text{ and }a\in \mathcal{A}.
    \end{equation}
        
    \begin{proof}[Sketched proof of \eqref{UIinequality}]
    As in \cite[Proposition 5.3]{KS-00}, we combine two observations:

    \begin{enumerate}
        \item[(1)] There exists $N_1\in \mathbb{N}$, $r>0$ and $p_1>0$, such that
        \[P_{N_1}(y,B_Y(a,\varepsilon))\ge p_1\quad \text{for any }y\in B_Y(z,r).\]

        \item[(2)] There exists $N_2=N_2(r)\in \mathbb{N}$ and $p_2=p_2(r)>0$ such that 
        \[P_{N_2}(y,B_Y(z,r))\ge p_2\quad \text{for any }y\in Y.\]
    \end{enumerate}
    Then \eqref{UIinequality} can be derived from the Kolmogorov--Chapman relation. 
    
    It remains to check (1) and (2). Thanks to the assumption that $S(z,\zeta)=z$ for some $\zeta\in \mathcal E$, it is easy to see that $\mathcal{A}_n(\{z\})$ is increasing, and that $\mathcal{A}\subset B_Y(\mathcal{A}_{N_1}(\{z\}),\varepsilon/2)$ for some $N_1\in \mathbb{N}$. By uniformly continuity, there exists $r=r(\varepsilon)>0$ and $\zeta_0^a,\dots ,\zeta_{N_1-1}^a\in \mathcal E$ such that 
    \begin{equation*}
        d(S_{N_1}(y;\zeta_0,\dots,\zeta_{N_1-1}), a)<\varepsilon,
    \end{equation*}
    provided $y\in B_Y(z,r)$ and $d_{\mathcal E}(\zeta_k , \zeta_k^a)<r$ for $k=0,\dots, N_1-1$. Thus for any $y\in B_Y(z,r)$,
    \begin{equation*}
        P_{N_1}(y,B_Y(a,\varepsilon))\ge \mathbb{P}(d_{\mathcal E}(\xi_k , \zeta_k^a)<r\text{ for } k=0,\dots, N_1-1)\ge (\inf_{\zeta\in \mathcal E}\mathbb{P}(d_{\mathcal E}(\xi_0, \zeta)<r))^{N_1}.
    \end{equation*}
    Noticing that the function $\zeta\mapsto \mathbb{P}(d_{\mathcal E}(\xi_0 , \zeta)<r)$ is lower semicontinuous and strictly positive in $\mathcal E$, we obtain (1). Finally, (2) is a consequence of \eqref{irreducibility} in hypothesis \hyperlink{I}{$\color{black}\mathbf{(I)}$}.
    \end{proof}

    \vspace{3mm}
    \noindent{\it Step 2. Uniform Feller property.} We only discuss the uniform Feller property on $Y$. The similar property on $\mathcal{A}$ follows in the same manner. That is to say, $\|Q_n^V \mathbf{1}\|_\infty^{-1}Q_n^V f$ is equicontinuous for any $f\in L(Y)$. Here we use a coupling method, which is analogous to \cite[Theorem 3.1]{JNPS-15}, while the setting there is different from ours. Recalling \eqref{g-def}, denote with
    \begin{equation}\label{def_delta1}
        \delta_1:=-\limsup\limits_{k\rightarrow\infty}\frac{1}{k}\log g(q^k)\in (0,\infty].
    \end{equation}
    We state the following lemma, which directly implies the uniform Feller property.

    \begin{lemma}\label{lemma_uf}
        For any $L>0$ and $\delta\in (0,\delta_1)$, there exists a continuous function $G\colon [0,1]\to [0,\infty)$ with $G(0)=0$, such that for any $V\in \mathcal{V}_{L,\delta}$, $f\in L(Y)$, $x,x'\in Y$ with $d(x,x')\le 1$, and $n\in \mathbb{N}$,
        \begin{equation}\label{UF inequality}
            |Q^V_nf(x)-Q^V_nf(x')|\le \|f\|_{L}\|Q^V_n\mathbf{1}\|_{L^\infty(Y)}G(d(x,x')).
        \end{equation}
    \end{lemma}
    
    \begin{proof}[Proof]
    Recall \eqref{squeezing} gives rise to a coupling operator on $Y\times Y$:
    \begin{equation*}
        \mathcal{R}(x,x'):=  (R(x,x'),R'(x,x'))\quad\text{for }x,x'\in Y.
    \end{equation*}
    Let $(\Omega_n,\mathcal{F}_n,\mathbb{P}_n)$ be a sequence of copies of the probability space on which $\mathcal{R}$ is defined, and $(\boldsymbol{\Omega},\boldsymbol{\mathcal{F}},\mathbb{P})$ the product probability space of this sequence. For any $x,x'\in X$ and $\boldsymbol{\omega}\in\boldsymbol{\Omega}$, we recursively define $(x_n,x_n')$ by 
    \begin{align*}
        (x_n(\boldsymbol{\omega}),x_n'(\boldsymbol{\omega}))&=\mathcal{R}^{\omega_n}(x_{n-1},x_{n-1}')\quad\text{for } n\in \mathbb{N},
    \end{align*}
    where $(x_0,x_0'):=(x,x')$.
    In particular, the laws of $x_n$ and $x'_n$ coincide with $P_n(x,\cdot)$ and  $P_n(x',\cdot)$, respectively. Let us denote by  $\mathbb{P}_{\vec{\boldsymbol{x}}}$ the law of the process issuing from $\vec{\boldsymbol{x}}=(x,x')$, and by $\mathbb{E}_{\vec{\boldsymbol{x}}}$ the corresponding expectation.

    For any $k\in \mathbb{N}_0$, we introduce the following events 
    \begin{align*}
        A_k=\{d(x_{k+1},x_{k+1}')\le qd(x_k,x_k')\},\quad B_k=\bigcap_{j=0}^k A_j,\quad C_k=\bigcap_{j=0}^{k-1} A_j \cap A_k^C.        
    \end{align*}
    Then for any $n\in \mathbb{N}_0$, we can decompose $\boldsymbol{\Omega}$ as the disjoint union
    \begin{equation*}
        \boldsymbol{\Omega}=\bigcup_{k=0}^{n-1} C_k \cup B_n.
    \end{equation*}
    In view of this partition, we now rewrite the left-hand side of \eqref{UF inequality} as follows.  
    \begin{align}
        &Q^V_nf(x)-Q^V_nf(x')=\mathbb{E}_{\vec{\boldsymbol{x}}}\left(f(x_n)e^{V(x_1)+\cdots+V(x_n)}-f(x'_n)e^{V(x'_1)+\cdots+V(x'_n)}\right)\notag\\
        =&\sum_{k=0}^{n-1} \mathbb{E}_{\vec{\boldsymbol{x}}}( \mathbf{1}_{C_k}(F_n(x)-F_n(x')))+\mathbb{E}_{\vec{\boldsymbol{x}}}( \mathbf{1}_{B_n}(F_n(x)-F_n(x')))=:\sum_{k=0}^{n-1} J_n^{(k)}+J_n',\label{UF inequality1}
    \end{align}
    where we adopt the notation $F_n(x):=f(x_n)e^{V(x_1)+\cdots+V(x_n)}$. Observe that:
   
    \begin{enumerate}
        \item[(1)] Using the Markov property, $C_k\in \mathcal{F}_{k+1}$ and $Q_{n-k-1}^V \mathbf{1}(x)\le e^{(k+1)\inf_X V}Q_n^V \mathbf{1}(x)$, we get
        \begin{align*}
            \mathbb{E}_{\vec{\boldsymbol{x}}} (\mathbf{1}_{C_k}F_n(x))&\le \|f\|_\infty \|Q_{n-k-1}^V \mathbf{1}\|_{L^\infty(Y)}\mathbb{E}_{\vec{\boldsymbol{x}}} (\mathbf{1}_{C_k}e^{V(x_1)+\cdots+V(x_{k+1})})\\
            &\le \|f\|_{\infty}\|Q_n^V\mathbf{1}\|_{L^\infty(Y)}e^{(k+1)\Osc(V)}\mathbb{P}_{\vec{\boldsymbol{x}}}(C_k).
        \end{align*}
        Here we tacitly used (as $\vec{\boldsymbol{x}}_n$ is a coupling of $P_n(x,\cdot)$ and $P_n(x',\cdot)$)
        \[\mathbb{E}_{\vec{\boldsymbol{x}}_{k+1}} e^{V(x_{k+2})+\cdots +V(x_n)}=\mathbb{E}_{x_{k+1}} e^{V(x_{k+2})+\cdots +V(x_n)}=Q_{n-k-1}^V \mathbf{1}(x_{k+1}).\]
        It can be seen, thanks to \eqref{squeezing}, that $\mathbb{P}_{\vec{\boldsymbol{x}}}(C_k)\le g(q^kd(x,x'))$. Hence for $k=0,\dots ,n-1$,
        \[|J_n^{(k)}|\le 2\|f\|_\infty \|Q_n^V \mathbf{1}\|_{L^\infty (Y)} e^{(k+1)\Osc(V)} g(q^kd(x,x')).\]

        \item[(2)] Rewrite $J_n'$ as $J_n'=J_{n,1}'+J_{n,2}'$, where
        \begin{align*}
            J_{n,1}'&=\mathbb{E}_{\vec{\boldsymbol{x}}}\left(\mathbf{1}_{B_n}(f(x_n)-f(x'_n))e^{V(x_1)+\cdots+V(x_n)}\right)\\    J_{n,2}'&=\mathbb{E}_{\vec{\boldsymbol{x}}}\left(\mathbf{1}_{B_n}f(x'_n)\left(e^{V(x_1)+\cdots+V(x_n)}-e^{V(x'_1)+\cdots+V(x'_n)}\right)\right).
        \end{align*} 
        Noticing that $d(x_k,x_k')\le q^kd(x,x')$ on $B_n$, we obtain 
        \[|J_{n,1}'|\le \Lip(f)  \|Q_n^V\mathbf{1}\|_{L^\infty(Y)}q^nd(x,x')\le  \Lip(f)\|Q_n^V\mathbf{1}\|_{L^\infty(Y)}d(x,x').\]
        Moreover, on $B_n$ we have
        \[\sum_{k=1}^n |V(x_k)-V(x_k')|\le \sum_{k=1}^n \Lip(V) q^k d(x,x')\le (1-q)^{-1} \Lip (V)d(x,x').\]
        Therefore
        \begin{align*}
            |J_{n,2}'|&=\left|\mathbb{E}_{\vec{\boldsymbol{x}}}\mathbf{1}_{B_n}f(x'_n)e^{V(x'_1)+\cdots+V(x'_n)}\left(e^{\sum_{k=1}^n (V(x_k)-V(x_k'))}-1\right)\right|\\
            &\le \|f\|_{\infty}\|Q_n^V\mathbf{1}\|_{L^\infty(Y)}\left(e^{(1-q)^{-1}\Lip(V)d(x,x')}-1\right).
        \end{align*}
    \end{enumerate}
    
    Plugging (1) and (2) into \eqref{UF inequality1}, we find that \eqref{UF inequality} holds with
    \begin{equation*}
        G(s):=2\sum_{k=0}^{\infty}e^{(k+1)\delta} g(q^ks)+s+e^{(1-q)^{-1}Ls}-1,\quad 0\le s\le 1.
    \end{equation*}
    Clearly $G(0)=0$. And $G$ converges uniformly in $[0,1]$ by \eqref{def_delta1} and $\delta<\delta_1$.
    \end{proof}
    
    Noting that so far we have verified the assumptions in Theorem~\ref{Thm Feynman-Kac}.

    \vspace{3mm}
    \noindent{\it Step 3. Concentration near $\mathcal{A}$.} This is a consequence of exponential mixing \eqref{Exponential mixing}, which can be found in \cite[Lemma 3.2]{MN-20}. As mixing takes place in $X=X_R$, we can prove a stronger result:
    \begin{equation}
        \lambda_V^{-n} \lim_{n\to \infty} \sup_{x\in X} Q_n^V(x,X\setminus B_X(\mathcal{A},r))=0,\label{concentration_estimate}
    \end{equation}
    provided $\Osc(V)<\gamma$, where $\gamma$ is the mixing rate. 

    \begin{proof}[Proof of \eqref{concentration_estimate}]
    For any $r\in (0,1)$, consider 
    \[h(x)=\frac{1}{r}\min\{\dist_X(x,\mathcal{A}),r\}\in L_b(X).\] Then $\|h\|_L=1+1/r$. In view of \eqref{Exponential mixing}, $\lambda_V\ge e^{\inf_X V}$ and $\langle h,\mu_*\rangle=0$, we find
    \begin{align}
        \sup_{x\in X} \lambda_V^{-n} Q_n^V(x,X\setminus B_X(\mathcal{A},r))&\le e^{n\Osc(V)}P_n(x,X\setminus B_X(\mathcal{A},r))\le e^{n\Osc(V)} \langle h,P_n^*\delta_x\rangle\notag\\
        &\le e^{n\Osc(V)} Ce^{-\gamma n}\|h\|_L=C(r)e^{-(\gamma-\Osc(V))n}.\label{ACverify}
    \end{align}
    If $\Osc(V)<\gamma$, then \eqref{concentration_estimate} follows immediately from this estimate.
    \end{proof}

    \vspace{3mm}
    \noindent{\it Step 4. Exponential boundedness.} Theorem~\ref{Thm Feynman-Kac} already yields
    \[M_1:=\sup_{n\in \mathbb{N}} \lambda_V^{-n}\|Q_n^V \mathbf{1}\|_{L^\infty(\mathcal{A})}<\infty.\]
    It remains to address
    \begin{equation}
        M:=\sup_{n\in \mathbb{N}} \lambda_V^{-n}\|Q_n^V \mathbf{1}\|_{L^\infty(Y)}<\infty,\label{eb_Y}
    \end{equation}
    provided $V\in \mathcal{V}_{L,\delta}$ with $\delta$ sufficiently small.  
    
    \begin{proof}[Proof of \eqref{eb_Y}]
    This proof is from \cite[Lemma 3.3]{MN-20}. Fix any $\delta_2<\delta_1$. Let us assume $\delta\le \delta_2$. By uniform Feller property \eqref{UF inequality}, there exists $r=r(L,\delta_2)>0$ such that
    \begin{equation}
        \|Q_n^V \mathbf{1}\|_{L^\infty(B_{Y}(\mathcal{A},r))}\le \|Q_n^V \mathbf{1}\|_{L^\infty(\mathcal{A})}+\frac{1}{4}\|Q_n^V \mathbf{1}\|_{L^\infty(Y)}.\label{EBverieq1}
    \end{equation}
    In view of the derivation of \eqref{ACverify}, for any $y\in Y$,
    \begin{equation}
        P_n(y,Y\setminus B_{Y}(\mathcal{A},r))\le C(r)e^{-\gamma n}.\label{EBverieq2}
    \end{equation}
    Let $\tau$ be the first hitting time of $B_{Y}(\mathcal{A},r)$ from $y\in Y$. By \eqref{EBverieq2}, if $\delta=\delta(r)$ is small enough, then $\mathbb{E}e^{\delta \tau}\le 2$. We claim that this $\delta$ is all we need, up to a harmless assumption that $\delta<\gamma$.
    
    In fact, by the strong Markov property, \eqref{EBverieq1} and \eqref{EBverieq2}, we have
    \begin{align*}
        |Q_n^V \mathbf{1}(y)|&\le e^{n\sup_X V}\mathbb{P}(\tau>n)+\sum_{k=0}^n \mathbb{E} (e^{V(y_1)+\cdots+V(y_n)}|\tau=k)\\
        &\le e^{n\sup_X V} P_n(y,Y\setminus B_{Y}(\mathcal{A},r))+\sum_{k=0}^n e^{k\sup_X V} \mathbb{P}(\tau=k) \|Q_{n-k}^V \mathbf{1}\|_{L^\infty(B_{Y}(\mathcal{A},r))}\\
        &\le C(r) e^{-n(\gamma-\sup_X V)}+\left(M_1 \lambda_V^n+\frac{1}{4}\|Q_n^V \mathbf{1}\|_{L^\infty(Y)}\right)\mathbb{E} e^{\Osc(V) \tau}\\
        &=(C(r)+2M_1) \lambda_V^n+\frac{1}{2}\|Q_n^V \mathbf{1}\|_{L^\infty(Y)}.
    \end{align*}
    Taking supremum on the left-hand side over all $y\in Y$, we obtain
    \[\|Q_n^V \mathbf{1}\|_{L^\infty(Y)}\le (C(r)+2M_1)\lambda_V^n+\frac{1}{2}\|Q_n^V \mathbf{1}\|_{L^\infty(Y)}.\]
    As a result, we find $M\le 2(C(r)+2M_1)<\infty$.
    \end{proof}

     This concludes the proof of Proposition~\ref{Prop Feynman-Kac stability}.

\vspace{4mm}

\noindent\textbf{Acknowledgments.}
The authors would like to thank Jia-Cheng Zhao for valuable discussions. 
Shengquan Xiang is partially  supported by NSFC 12301562 and by “The Fundamental Research Funds for the Central Universities, 7100604200, Peking University”. Zhifei Zhang is  partially supported by  NSFC 12288101. 

    \normalem
    \bibliographystyle{alpha}
    \bibliography{main}
    
\end{document}